\documentclass[11pt,reqno]{amsart}
\usepackage{amsmath,amsthm,amssymb,amsfonts,bbm,mathtools,dsfont}
\usepackage{enumerate,hyperref,mathtools}
\usepackage{young}
\usepackage{graphicx}
\usepackage{float}
\usepackage[mathscr]{eucal}
\usepackage[margin=1.2in]{geometry}
\usepackage{microtype}
\usepackage{hyperref}
\usepackage{pslatex}

\newtheorem{maintheorem}{Theorem}
\newtheorem{theorem}{Theorem}[section]
\newtheorem{corollary}[theorem]{Corollary}
\newtheorem{lemma}[theorem]{Lemma}
\newtheorem{remark}[theorem]{Remark}

\newtheorem{proposition}[theorem]{Proposition}
\newtheorem{definition}[theorem]{Definition}

\newcommand{\D}{\Delta}
\newcommand{\al}{\alpha}

\newcommand{\p}{\mathbb{P}}
\renewcommand{\P}{\mathbb{P}}
\newcommand{\E}{\mathbb{E}}

\def\T{{\mathcal T}}
\def\M{{\mathcal M}}
\def\Tree{{ T}}

\def\sTV{\mbox{\tiny\rm TV}}

\def\M{\mathcal{M}}

\def\mm{\mathfrak{m}}

\def\limMes{{\bar \nu}}

\newcommand{\defeq}{\vcentcolon=}

\begin{document}

\title{Decay of Correlations for the Hardcore Model on the $d$-regular Random Graph}

\author{Nayantara Bhatnagar}

\address{Department of Mathematical Sciences,  University of Delaware, Newark, DE, 19716.}
\email{nayantara.bhatnagar@gmail.com}

\author{Allan Sly}
\address{Department of Statistics, University of Califoria, Berkeley, Berkeley, CA, 94720.}
\email{sly@stat.berkeley.edu}

\author{Prasad Tetali}
\address{School of Mathematics and School of Computer Science, Georgia Institute of Technology, Atlanta, GA 30332.}
\email{tetali@math.gatech.edu}

%
%
%

\maketitle

\begin{abstract}
A key insight from statistical physics about spin systems on random graphs is the central role played by Gibbs measures on trees.
We determine the local weak limit of the hardcore model on random regular graphs asymptotically until just below its condensation threshold, showing that it converges in probability locally in a strong sense to the free boundary condition Gibbs measure on the tree.  As a consequence we show that the reconstruction threshold on the random graph, indicative of the onset of point to set spatial correlations,  is equal to the reconstruction threshold on the $d$-regular tree for which we determine precise asymptotics. We expect that our methods will generalize to a wide range of spin systems for which the second moment method holds.
\end{abstract}

\section{Introduction}

In this paper we consider the hardcore model on random $d$-regular graphs and study its local spatial mixing properties.  We determine the location of a phase transition where the  model undergoes a spatial mixing transition after which the spin at a typical vertex becomes dependent over long distances.  Theory from statistical physics relates this transition to the clustering or shattering threshold and both  of these transitions appear to be related to the apparent computational difficulty of finding large independent sets.  No algorithms are known to find independent sets of size $\frac{(1+\epsilon)\log d}{d} n$ in a random $d$-regular graph on $n$ vertices which coincides with the spatial mixing threshold.  In contrast the maximal independent set is of size $\frac{(2-o_d(1))\log d}{d} n$ \cite{FriLuc:92}.  In this work, we show that the reconstruction or extremality threshold on the infinite $d$-regular tree determines the onset of long distance point to set spatial correlations in the random $d$-regular graph.  We prove an asymptotic lower bound on the reconstruction  threshold which matches the known upper bound in the first two terms of the asymptotic series. Together, these results determine the asymptotic location of the threshold for the random $d$-regular graph for the onset of point to set correlations over long distances.

For a finite graph $G=(V,E)$ an independent set is a subset of the
vertices containing no adjacent vertices. Denote the set of independent sets as $I(G)$. We will view an independent set as a {\em spin configuration} $\sigma$, taking values in $\{0,1\}^V$ with $\sigma_v$ denoting the spin at the vertex $v$. The {\em hardcore model} (or {\em hardcore measure}) is the
probability measure over the set of independent sets $\sigma \in I(G)$ given by
\begin{equation}\label{eq:defnHardcore}
\p(\sigma)=\frac1{Z} \lambda^{\sum_{v\in V}\sigma_v} \mathbbm{1}_{\sigma\in I(G)}.
\end{equation}
The parameter $\lambda>0$ is known as the \emph{fugacity} and controls the typical size of an independent set with larger values of $\lambda$ putting more of the weight of the distribution on larger independent sets. As usual, $Z$ is a
normalizing constant called the partition function.
The definition of Gibbs measures and the hardcore model in particular can be extended to infinite graphs by way of the Dobrushin-Lanford-Ruelle condition
which essentially says that for every finite set $A$, the probability of a configuration on $A$ is given by the Gibbs distribution given by a random boundary generated by the measure outside of $A$.  Such a measure is called a Gibbs measure, and it may not be unique
(see e.g. \cite{Geo88} for more details).


On the infinite $d$-regular tree $\Tree_d$, there is a unique Gibbs measure for the hardcore model if and only if $\lambda \leq \frac{(d-1)^{d-1}}{(d-2)^d}$.  However, for every $\lambda$, there exists a translation
invariant Gibbs measure given by a Markov model on the tree which we denote by $\P_{\Tree_d}$ (henceforth, we refer to this as ``the translation invariant measure" on $T_d$).
We denote the density of $\P_{\Tree_d}$, that is, the probability that a site is occupied, by $\alpha = \alpha(\lambda,d)$ which satisfies the relation
\begin{equation}\label{e:alphaLambdaRelation}
\lambda=\frac{\al}{1-2\al}\left(\frac{1-\al}{1-2\al}\right)^{d-1}.
\end{equation}
Since $\alpha=\alpha(\lambda,d)$ is a strictly monotone increasing function of $\lambda$  we will use both parameters interchangeably to specify the model depending on the context.  The density of the largest independent set of a $d$-regular random graph is asymptotically $(2\log d - (2+o(1))\ln \ln d)/d$ \cite{FriLuc:92}.  The results we present hold very close to this threshold, up to
\[
\alpha< \alpha_c(d) := \frac{(2\log d - (3+o(1))\ln \ln d)}{d}.
\]
We take $\lambda_c$ to be the corresponding value of $\lambda$.
The bulk of this paper is devoted to establishing that the hardcore measure on the random $d$-regular graph is well approximated locally by the measure $\P_{\Tree_d}$ when $\lambda < \lambda_c$.  We prove that the measure converges in a strong notion of local weak convergence described in Section~\ref{s:localWeakConv}.

\begin{maintheorem}\label{thm:local-weak-conv}
Let $G_n$ be the random $d$-regular graph on $n$ vertices.  Then for large enough $d$, the hardcore measure on $G_n$ with fugacity $\lambda  < \lambda_c $ converges in probability locally to the measure $\P_{\Tree_d}$.
\end{maintheorem}

Our methods provide a general framework for proving convergence in probability locally which we expect will apply to various other Gibbs measures on random graphs such as colorings or NAE-SAT.
Having established Theorem~\ref{thm:local-weak-conv}, it is natural to consider properties of the measure $\P_{\Tree_d}$.
The set of Gibbs measures is convex and so we may ask whether $\P_{\Tree_d}$ is extremal, that is, it is not a convex combination of other Gibbs measures.  Extremality is equivalent to a notion of point to set correlation on trees called the reconstruction problem (for a survey, see~\cite{Mos04}).

To formalize the definition of the problem, we will make use of a description of $\P_{\Tree_d}$ as a
Markov model on the tree generated as follows.
First the spin at the root is chosen to be occupied with probability $\alpha$ and unoccupied with probability $1-\alpha$,
where $\alpha$ is chosen as in~\eqref{e:alphaLambdaRelation}.
The spins of the remaining vertices of the graph are generated from their parents' spins by taking one step of the Markov transition matrix
$$M = \left(
\begin{array}{cc}
p_{11} & p_{10}\\
p_{01} & p_{00}
\end{array}
\right)
=
\left(
\begin{array}{cc}
0 & 1\\
\frac{\alpha}{1- \alpha} & \frac{1-2\alpha}{1-\alpha}\\
\end{array}
\right)\,,
$$
where $p_{i  j}$ denotes the probability of the spin at a vertex being $j$ given that the spin of the parent is state $i$.
Since $(\alpha,1-\alpha)$ is reversible with respect to $M$ this gives a translation invariant measure on $T_d$ which corresponds to the measure $\P_{\Tree_d}$ with fugacity $\lambda$.

Let $\sigma(L)$ denote the spins of the vertices at distance $L$ from the root as generated by the Markov model described above.
The reconstruction problem on the tree asks if we can
recover information on $\sigma_\rho$, the spin of the root $\rho$ from the spins $\sigma(L)$ as $L \to \infty$.
Formally, we say that the model  $(\Tree_d,M)$ has \emph{non-reconstruction} if
\begin{eqnarray}\label{eq:tree-non-recon}
\lim_{L \to \infty} \; \P_{\Tree_d}(\sigma_\rho=1|\sigma(L))\to \alpha(\lambda,d)
\end{eqnarray}
in probability as $L \to \infty$, and otherwise, the model has
\emph{reconstruction}. Non-reconstruction is equivalent to extremality of the Gibbs measure or that the tail
$\sigma$-algebra of the Gibbs measure is trivial~\cite{Mos04}.


Information theoretically, non-reconstruction corresponds to fast decay of correlations between the spin at the root and the spins of far away vertices \cite{Mos04}. Proposition~12 of~\cite{Mos01} implies that there exists a critical fugacity
$\lambda_R$ (or, equivalently, a critical density $\alpha_R$) such that reconstruction holds for the hardcore model with fugacity $\lambda > \lambda_R$
and non-reconstruction holds for $\lambda < \lambda_R$.  The reconstruction problem on the tree was originally studied as a
problem in statistical physics but has since found many applications
including in computational phylogenetic reconstruction~\cite{DasMosRoc06}, the
study of the geometry of the space of random constraint satisfaction
problems (CSP's) \cite{AchCoj08,KrzMonRicSemZde07} and the mixing time of Markov
chains~\cite{BerKenMosPer05, BorChaMosRoc06,MarSinWei04,ResSteVerVigYan11,TetVerVigYan10}.

Here we establish tight bounds on the reconstruction threshold for the hardcore model on the $d$-regular tree\footnote{This result previously appeared in extended abstract form in \cite{BhaSlyTet10}.}.  The upper bound was shown by Brightwell and Winkler~\cite{BriWin04}, and our contribution is the lower bound.
\begin{maintheorem}\label{thm:tree-non-reconst}
For large enough $d$, the reconstruction threshold for $\P_{\Tree_d}$ on the $d$-regular tree satisfies
\[
\frac{(\ln 2-o(1))  \ln^2 d}{2 \ln \ln d} \leq \lambda_R \leq (e+o(1))\ln^2 d.
\]

\end{maintheorem}

Prior to our work, Martin~\cite{Mar03} had shown that $\lambda_R > e-1$. Restating  Theorem~\ref{thm:tree-non-reconst} in terms of $\al$ we have
that the critical density for reconstruction satisfies
\begin{equation}\label{e:omegaBound}
\frac1d \left( \ln d + \ln
  \ln d - \ln \ln \ln d  - \ln 2 +  \ln \ln 2 - o(1) \right)\leq \al_R \leq \frac1d \left( \ln d + \ln
  \ln d +1+o(1) \right)
\end{equation}
leaving only an additive $(\ln \ln \ln d)/d$ gap between the bounds. The form of our bound in equation \eqref{e:omegaBound} is
strikingly similar to the bound for the $q$-coloring model~\cite{Sly09a} which
states that reconstruction (resp. non-reconstruction) holds when the
degree $d$ is at least (resp. at most) $q(\ln q + \ln \ln q + O(1))$.

The next theorem, combined with Theorem \ref{thm:tree-non-reconst} gives a precise picture of the local spatial mixing properties of the hardcore model on the random $d$-regular graph.  In~\cite{GerMon07} a natural extension of the reconstruction problem was introduced for graphs.
Let $\{G_n\}$ be a family of random graphs  whose size $n$ goes to infinity, and let $\sigma$ be distributed according to the hardcore model with fugacity $\lambda$. We will use $\sigma(S)$ to denote the configuration on a subset of vertices $S$ and $\sigma_v$ to denote the spin at a vertex $v$. The model has non-reconstruction if for a uniformly chosen $u \in V(G_n)$,
\begin{align}\label{eq:graph-non-recon}
\lim_{L \to \infty} \; \limsup_n \; \E \left|\P\Big(\sigma_u=1 | \sigma(\partial B_u(L)), u\Big) - \alpha(\lambda,d)\right| = 0
\end{align}
where $B_u(L)$ denotes the vertices within distance $L$ of $u$ (and by abuse of notation, the induced subgraph), $\partial B_u(L)$ denotes the boundary of $B_u(L)$ and $\al(\lambda,d)$ is the density given by~\eqref{e:alphaLambdaRelation}.

\begin{maintheorem}\label{thm:graph-tree-reconst}
Let $\lambda < \lambda_c$ and let $\al(\lambda,d)$ be the density given by~\eqref{e:alphaLambdaRelation}. Let $G_n$ be the random $d$-regular graph on $n$ vertices and let $u$ be a uniformly random vertex in $V(G_n)$. Then, for large enough $d$,
\begin{align*}
\P_{T_d}(\sigma_\rho =1 | \sigma(L))  &\stackrel{\P}{\to} \alpha(\lambda,d) \ \ \text{as}\;  L \to \infty \\
& \Leftrightarrow \nonumber \\
\lim_{L \to \infty} \; \limsup_n \; \E \bigg|\P\Big(\sigma_u &=1 | \sigma(\partial B_u(L)), U\Big)  - \alpha(\lambda,d)\bigg| = 0
\end{align*}
\end{maintheorem}
That is, the random $d$-regular graph has non-reconstruction if and only if $(T_d,M)$ has non-reconstruction.




\subsection{Related work}
A significant body of work has been devoted to the reconstruction problem on the $d$-regular tree by probabilists, computer scientists and physicists for a number of different spin configuration models.  The earliest such result is the Kesten-Stigum bound~\cite{KesSti66} which states that for a Markov model defined on the tree, reconstruction holds whenever $\theta^2 (d-1) > 1$, where $\theta$ is the second largest eigenvalue of the corresponding Markov matrix.  This bound was shown to be tight in the case of the Ising model~\cite{BleRuiZag95,EvaKenPerSch00} where it was shown that non-reconstruction holds when $\theta^2 (d-1) \leq 1$.  Similar results were derived for the  Ising model with small external field~\cite{BorChaMosRoc06} and the 3-state Potts model~\cite{Sly09} which constitute the only models for which exact thresholds are known.  On the other hand, for the hardcore model $\theta^2 (d-1)= (1+o(1))\frac1d \ln^2 d$ and thus at least when $d$ is large, the Kesten-Stigum bound is known not to be tight~\cite{BriWin04}.


In both the coloring model and the hardcore model the reconstruction
threshold is far from the Kesten-Stigum bound for large $d$.
In the coloring
model close to optimal bounds on the reconstruction threshold~\cite{BhaVerVigWei11,Sly09a} were obtained by
first showing that, when $n$ is small, the information on the root is
sufficiently small.
Then  a quantitative version
of~\cite{JanMos03} establishes that the information on the root converges to
0 exponentially quickly.  In this work, we show that the hardcore model behaves similarly.

\subsubsection{Replica Symmetry Breaking and Finding Large Independent Sets}

The reconstruction problem plays a deep role in the geometry of the
space of solutions of random CSPs.  While
for problems with few constraints the space of solutions is connected
and finding solutions is generally easy, as the number of constraints
increases the space may break into exponentially many small clusters.
Physicists, using powerful but non-rigorous ``replica symmetry
breaking'' heuristics, predicted that the clustering phase transition
exactly coincides with the reconstruction region on the associated
tree model~\cite{MezMon06,KrzMonRicSemZde07}.
This picture was rigorously established (up to first order terms) for the coloring and
satisfiability problems~\cite{AchCoj08} and further extended to sparse
random graphs by~\cite{MonResTet11}.
When solutions are far apart,
local search algorithms will in general fail. Indeed for both the
coloring and SAT models, no algorithm is known to find solutions in
the clustered phase.  It has been conjectured to be computationally
intractable beyond this phase transition \cite{AchCoj08}.

Previous results \cite{GerMon07,MonResTet11} have related the reconstruction problem on the Poisson tree with constant expected degree with
reconstruction in sparse random graph ensembles. These results
established a ``replica'' condition saying that
the empirical distribution of pairs of spins at a vertex from two
independent configurations are from a product measure.  This does not
apply in the case of the hardcore model  since the degree of a vertex affects its probability of being in
the independent set. At the same time, for the $d$-regular random graph the methods of \cite{MonResTet11} do not seem to be directly applicable and we approach the problem instead using the theory of local weak convergence of Gibbs measures.

The associated CSP for the hardcore model  corresponds to finding
large independent sets in random $d$-regular graphs. The replica symmetric
heuristics again predict that the space of large independent sets
should be clustered in the reconstruction regime.  Specifically this
refers to independent sets of size $\alpha n$ where $\alpha > \alpha_R$, the
density of 1's in the hardcore model at the reconstruction threshold, and roughly half the density of the largest independent set~\cite{FriLuc:92}. On the other hand the best
known algorithm finds independent sets only of density $\frac{(1+o(1))\ln
  d}{d}$ which is equal to $\alpha_R $ asymptotically as $d \to \infty$~\cite{Wor95}.  This is
consistent with the physics predictions and it was shown that on Erd\H{o}s-Renyi random graphs, independent sets exhibit the same
clustering phenomena~\cite{CojEft:11} as colorings and SAT \cite{AchCoj08, KrzMonRicSemZde07} at the reconstruction
threshold and one would expect this to also be the case for random regular graphs.
By determining the reconstruction threshold on such graphs we provide further evidence supporting the computational hardness of finding large independent sets in random graphs.

Sufficiently close to the satisfiability threshold many CSPs including the hardcore model are believed to undergo an additional  phase transition called the condensation~\cite{KrzMonRicSemZde07}.  Beyond this transition the second moment method fails and the distribution places most of its weight on a constant number of clusters \cite{BapCojHetRasVil14}.  After the condensation transition it is believed that the hardcore measure no longer converges locally to $\P_{\Tree_d}$ explaining the necessity of an upper bound on  $\lambda$ in our theorems.

\subsection{Local Weak Convergence}\label{s:localWeakConv}

There are a number of natural notions of local weak convergence of Gibbs measures and we introduce these now, following the notation used in~\cite{MonMosSly12}.
Let $\T_d$ denote the space of hardcore Gibbs measures on
$\Tree_d$ endowed with the topology of weak convergence
and let $\M_d$  to be the space
of probability measures over $\T_d$.
For a sequence of graphs $G_n$ we denote a hardcore measure by $\mu_n$
while $\nu$ denotes a hardcore measure on $\Tree_d$.
The notation $\Tree_d(L)$ will denote the restriction of the tree $\Tree_d$ to a ball of radius $L$ around the root (and by abuse of notation, we also use it to denote the set of vertices of the restriction). The shorthand $\mu_n^{L}$ or $\nu^{L}$ denote the restrictions of the corresponding measures to a ball of radius $L$.  
For a measure on Gibbs measures $\mm\in \M_d$, we let $\mm^L$
denote the measure on the space of measures on $\{0,1\}^{\Tree_d(L)}$ induced by such projections.

\begin{definition}
Consider a sequence of graphs-Gibbs measure
pairs  $\{(G_n,\mu_n)\}_{n\in\mathbb N}$ and for $v \in V(G_n)$, let
$\p_n^L(v)$ denote the law of the pair
$(B_v(L),\sigma(B_v(L)))$ when $\sigma$ is drawn with distribution $\mu_n$. Let $U_n$ denote the uniform measure over a random vertex $ u \in V(G_n)$.
Let $\p_n^L = \E_{U_n}(\p_n^L(u))$ denote the average of $\p_n^L(u)$. Let $\delta_{\Tree_d(L)}$ denote the Dirac measure on graphs which is $1$ on $\Tree_d(L)$.
\begin{enumerate}
\item[A.] The first mode of convergence concerns picking a random vertex $u$ and a random local configuration in the neighbourhood of $u$.  Formally, for $\limMes\in \mathcal{T}_d$ we say that $\{\mu_n\}_{n\in\mathbb N}$
{\em converges locally on average to} $\limMes$
if for any $L$,
\begin{eqnarray}
\lim_{n\to\infty} d_{\sTV}\left(\p_n^L, \delta_{\Tree_d(L)} \times \limMes^L \right) = 0.
\end{eqnarray}
\item[B.] A stronger form of convergence involves picking a random vertex $u$ and the associated random local measure $\p_n^L(u)$ and asking if this distribution of distributions converges. Formally, we say that {\em the local distributions of} $\{\mu_n\}_{n\in\mathbb N}$
{\em converge locally to} $\mathfrak{m}\in\mathcal{M}_d$ if it holds that the law of
$\p_n^L(u)$ converges weakly to $\delta_{\Tree_d(L)} \times \mathfrak{m}^L$ for all $L$.
\item[C.] If $\mathfrak{m}$ is a point mass on $\limMes\in \mathcal{T}_d$ and if the local distributions of $\{\mu_n\}_{n\in\mathbb N}$
converge locally to $\mathfrak{m}$
then we say that $\{\mu_n\}_{n\in\mathbb N}$
{\em converges in probability locally to} $\limMes$.  Equivalently convergence in probability locally to $\limMes$ says that for any $L$ and any $\varepsilon > 0$ it holds that
\begin{eqnarray}
\lim_{n\to\infty} U_n \left(d_{\sTV}(\p_n^L(u), \delta_{\Tree_d(L)} \times \nu^L \right) > \varepsilon) = 0.
\end{eqnarray}
\end{enumerate}
\end{definition}
\begin{remark}\label{rem:remarks-local-convergence}
As noted in \cite{MonMosSly12}, $C \Rightarrow B \Rightarrow A$ while in~\cite{SlySun12} it is noted that
if the measures $\nu$ are {\em extremal} Gibbs measures then the three notions of convergence $A$, $B$ and $C$ are equivalent.
\end{remark}

At a high level, convergence locally on average to $\nu$ means that after averaging the local distribution of configurations over all the vertices, the random configuration converges weakly to $\nu$ while convergence in probability locally to $\nu$ means that the local distribution at almost every vertex is close to $\nu$ eventually.
As noted above the former is a weaker condition and is in fact much simpler to prove. One can apply the second moment method for the hardcore model on the random $d$-regular graph for a large range of $\lambda$ to relate the hardcore measure to its planted version  where one first chooses a random independent set and then constructs a uniformly chosen  graph compatible with the set.  By exploring the graph in the planted measure by progressively revealing its edges one can show convergence locally on average to the measure $\P_{\Tree_d}$ and via the second moment method this can be extended to the original hardcore distribution.  This argument does not imply the stronger local convergence in probability and indeed, if one assumes the picture developed in
statistical physics, in the condensation phase one expects local convergence of type A but not convergence of type B or C.

In order to investigate the reconstruction problem it is necessary to work with local convergence in probability.
Much of the work of the paper involves showing how the second moment method can be used to imply  this stronger notion of convergence. Thus, our proof shows that for the hardcore model, up to the fugacity for which the second moment method holds, the notions A and C of local convergence of measures are equivalent. Our methods are quite general and should apply to a broad range of CSPs and Gibbs measures on graphs. Roughly speaking, one would need to show a corresponding bound on the second moment of the partition function and concavity of the log-partition function. One would also need to show that the partition function changes by a bounded amount when an edge is added and as such, our method should be applicable to non-zero temperature models.

\subsection{Outline of the proof}
We begin by establishing a lower bound on the reconstruction threshold for the $d$-regular tree in Section \ref{sec:tree-reconst}, proving Theorem~\ref{thm:tree-non-reconst}.  We show  that when $\alpha$ is bounded by the lower bound in \eqref{e:omegaBound} then even for a tree of depth
3 there is already significant loss of information of the spin at the
root.  In particular we show that if the spin of the root was 1 then the typical
posterior probability that the spin of the root is 1 given the spins at level 3
will be less than $\frac12$. The result is completed by linearizing
the  tree posterior probability recursion similarly to~\cite{BorChaMosRoc06,Sly09}.
In this part of the proof we closely follow the analysis
of~\cite{BorChaMosRoc06} who analyzed the reconstruction problem for the Ising
model with small external field. We do not require the full strength
of their analysis, however, as in our case we are far from the Kesten-Stigum
bound.
We show that a quantity referred to as the \emph{magnetization} decays
exponentially fast to~0. The magnetization provides a bound on the
posterior probabilities and this completes the result.

The $\ln \ln d$ term in our bound on $\lambda_R$ in Theorem \ref{thm:tree-non-reconst} is explained as the
first point at which there is significant decay of information at
level 3 on the tree.  In particular the analysis in
Proposition~\ref{prop:expected-small} part $c)$ is essentially tight.
It may be possible to get improved bounds by considering higher depth
trees although the description of the posterior distribution
necessarily becomes more complex.  A sharper analysis of this sort was
done in~\cite{Sly09a} for the coloring model although the method
there made crucial use of the symmetry of the states.

The bulk of the paper concerns proving local weak convergence to $\P_{\Tree_d}$ for the hardcore model on the random $d$-regular graph and this is shown in Theorem \ref{thm:graph-tree-reconst-equiv} in Section \ref{sec:local-weak-convergence}.  Our main tool is a new approach to the use of the second moment method. We select say, $n^{\frac 35}$ randomly chosen vertices in the $d$-regular random graph, and consider a ``punctured" graph with the local neighborhoods of these vertices removed. The punctured graph is used to study the partition function of the original graph conditional on the configuration of the boundaries of these neighborhoods. The second moment method in combination with Azuma's inequality implies that the partition function conditioned on a boundary configuration is within a multiplicative factor of $\exp(O(n^{\frac 12+\varepsilon}))$ of the expected partition function.
We prove convergence in probability locally by showing that it is extremely unlikely that a constant fraction of the $n^{\frac 35}$ randomly chosen vertices have a local measure which is far from the translation invariant measure on the tree. Indeed, we show that this would entail the existence of a set of configurations on the set of boundary vertices which has a constant probability under the hardcore measure but expected probability of only $\exp(-cn^{\frac 35})$.  In Proposition~\ref{prop:E1-E2} we show that this is precluded by the second moment method. 

One strength of our approach is that it does not require the detailed calculations of the small graph conditioning method. In many spin systems, including the one studied here, the ratio of the second moment of the partition function to the square of the first moment tends to a value $>1$ and so the second moment method cannot be used to estimate the partition function with probability tending to $1$. In this case, small graph conditioning can be used to give estimates on the partition function \cite{Wor99}.

The first and second moments of the hardcore partition function for a $d$-regular random graph are derived in Section \ref{sec:partition-fn-G} while the calculations for the punctured random graph appear in Section \ref{sec:tildeG}. The remaining proof involves establishing the requisite bound on the second moment itself.  This involves determining the maximum of a function which corresponds to the expected number of pairs of independent sets in a random regular graph with a given overlap between them. In Proposition \ref{prop:f-global-max}, which is proved in Section \ref{sec:tech-lemmas}, we consider the scaled log-partition function, determine its maximum and show that it decays quadratically near its maximum. This is a key fact used in relating the first and second moments of the partition functions of the random graph in Section \ref{sec:tildeG}.

\section{Upper bound on the reconstruction threshold on the tree}\label{sec:tree-reconst}

In this section we present the proof of Theorem~\ref{thm:tree-non-reconst}. We start by noting that for any finite restriction of $T_d$ to its first $n$ levels, we can use the Markov matrix $M$ as before to generate an independent set from the hardcore measure by setting the spin of the root to be occupied with probability $\alpha$ and then applying the matrix as before to generate the spins at the children recursively until we reach the leaves of the tree.

We define the following quantities which are related to the transition probabilities of the Markov matrix $M$. Let
\begin{eqnarray}
\nonumber \pi_{01} = \frac{1-\al}{\al}, \ \ \ \ \ \
\ \D \coloneqq \pi_{01} -1 = \frac{1-2\al}{\al}, \ \mathrm{and}
\\ \nonumber \theta \coloneqq
p_{00} -p_{10} = p_{11}-p_{01} = - \frac{\al}{1-\al}.
\end{eqnarray}

As mentioned in the introduction, $\theta$, the second
eigenvalue of $M$, plays a particularly important role in the reconstruction problem.

For ease of notation, we will establish non-reconstruction for the model $(\widetilde T_d,M)$ where $\widetilde T_d$ is the $d$-ary tree
(where each vertex has $d$ children) rather than on the $d$-regular tree.
It is not difficult to modify the recursion we will obtain for the
$d$-ary tree to a
recursion for the $(d+1)$-regular tree, showing that non-reconstruction
also holds in that case.
Finally, we can show that non-reconstruction on the $d$-regular tree
is equivalent to non-reconstruction on the $(d+1)$-regular tree once we
note that in equation \eqref{e:omegaBound} we have that
$\al_R(d+1)-\al_R(d) =
o(d)$ so the difference can be absorbed in the error term.
We will use $T$ to denote a finite tree whose root will be denoted $x$. 
Let  $\p^1_{T},\E^1_{T}$ (and resp. $\p^0_{T},\E^0_{T}$ and
$\p_{T},\E_{T}$) denote the probabilities and expectations with respect to the
measure on the leaves of $T$ obtained by conditioning on the root $x$ to be 1 (resp. 0, and
stationary). Let $L=L(n)$ denote the set of vertices of $T$ at depth $n$ and let
$\sigma(L)=\sigma(L(n))$ denote the configuration on level $n$.  We
will write $\p_{T}(\cdot|\sigma(L)=A)$ to denote the measure
conditioned on the leaves being in state $A\in\{0,1\}^{L(n)}$.

As in~\cite{BorChaMosRoc06} we analyze the  {\em
  weighted magnetization of the root} of $T$ which is a function of the random configuration the vertices at distance $n$ from the root and defined as follows:
\begin{eqnarray}\nonumber
X  = X(n) & := &(1-\alpha)^{-1}[(1-\alpha)\p_T(\sigma_x=1|\sigma(L)) - \alpha\p_T(\sigma_x=0|\sigma(L))] \\
\label{eq:mag2} & = &
\frac{1}{\pi_{01}}\left(\frac{\p_T(\sigma_x=1| \sigma(L))}{\alpha}-1\right).
\end{eqnarray}
Notice that since $\E_{T}(\p_T(\sigma_{x}=1| \sigma(L))) = \p_T(\sigma_x=1)=\alpha$, by \eqref{eq:mag2}, we
have that $\E_T(X)=0$. Also, from the first line of \eqref{eq:mag2}, it can be verified that $X \leq 1$ since $\p_T(\sigma_{x}=1|
  \sigma(L)=A) \leq
1$ for any $A$. We will also make use of the following second moments of the
magnetization.
\begin{eqnarray}\nonumber
\overline X = \overline X(n):=\E_{T}(X^2), \ \ \ \overline X_1 = \overline X_1(n):= \E_{T}^1(X^2), \ \ \
\overline X_0 = \overline X_0(n):= \E_{T}^0(X^2)
\end{eqnarray}
The following equivalent definition of non-reconstruction is
well known and follows from the definition in \eqref{eq:tree-non-recon} using
\eqref{eq:mag2}.
\begin{proposition}\label{p:bar-x-non-reconst} Non-reconstruction for
  the model $({\widetilde T_d},M)$ is equivalent
  to $$\lim_{n \rightarrow \infty} \overline X = 0.$$
\end{proposition}

In the remainder of the proof we derive bounds for $\overline
X$. We begin by showing that already for a 3 level tree, $\overline X$
becomes small.  Then we establish a recurrence along the lines of
\cite{BorChaMosRoc06} that shows that once $\overline X$ is sufficiently small,
it must converge to 0.  As this part of the derivation follows the
calculation in \cite{BorChaMosRoc06} we will adopt their notation in
places.  Non-reconstruction is then a consequence of
Proposition~\ref{p:bar-x-non-reconst}.  In the next lemma we
determine some basic properties of $X$.

\begin{lemma}\label{lem:basic-relations} For any $n \ge 1$, the following relations hold:
\begin{enumerate}[a)]
\item $\E_{T}(X) = \alpha\E^1_{T}(X) + (1-\alpha)\E^0_{T}(X).$
\item $\overline X = \alpha\overline X_1 + (1-\alpha) \overline X_0.$
\item $\E_{T}^1(X) = \pi_{01} \overline X$ and $\E_{T}^0(X) = - \overline X.$
\end{enumerate}
\end{lemma}

\begin{proof}
Note that for any random variable $Y = Y(A)$ which depends only on the states at
the leaves, we have $E_{T}(Y) = \alpha\E^1_{T}(Y) +
(1-\alpha)\E^0_{T}(Y)$. Parts $a)$ and $b)$ therefore follow since $X$ is a random
variable that is a function of the states at the leaves.
For part $c)$ we proceed as follows. The first and last equalities
below follow from \eqref{eq:mag2}.
\begin{eqnarray}
\nonumber \E_{T}^1(X) & = &
\pi_{01}^{-1}\displaystyle\sum_{A}\p_{T}(\sigma(L)=A|\sigma_x=1)\left(
\frac{\p_{T}(\sigma_x=1| \sigma(L)=A)}{\alpha}-1 \right) \\
\nonumber & = &
\pi_{01}^{-1} \displaystyle\sum_{A}\p_{T}(\sigma(L)=A)\frac{\p_{T}(\sigma_x=1| \sigma(L)=A)}{\alpha}\left(
\frac{\p_{T}(\sigma_x=1| \sigma(L)=A)}{\alpha} -1 \right) \\
\nonumber & = &
\pi_{01}^{-1}\left( \frac{\E_{T}((\p_{T}(\sigma_x=1| \sigma(L)))^2)}{\alpha^2} -1
\right) \\
\nonumber & = & \pi_{01}\E_T(X^2)
\end{eqnarray}
The second part of $c)$ follows by combining this with $a)$ and the fact that $\E_T(X)=0$.
\end{proof}

The following proposition estimates typical posterior probabilities
which we will use to bound $\overline X$.  Let $T(n)$ denote the tree which is the restriction of $\widetilde T_d$ to its first $n$ levels. For a finite tree $T$, let $T^i$ be the subtrees rooted at the
children of the root $u_i$.

\begin{proposition}\label{prop:expected-small}
For a finite $d$-ary tree $T$ we have that
\begin{enumerate}[a)]
\item For any configuration at the leaves
  $A=(A_1,\cdots,A_d)$,
\begin{eqnarray}
\nonumber \p_{T}(\sigma_x=0|\sigma(L)=A)=\Bigl({1+\lambda
  \prod_{i}\p_{T^i}(\sigma_{u_i}=0 | \sigma_{L_i} = A_i)}\Bigr)^{-1}.
\end{eqnarray}

\item Let $\mathcal G$ be the set of leaf
  configurations
\begin{eqnarray}
\nonumber \mathcal G = \left\{\sigma(L) \ | \ \p_T(\sigma_x=0|\sigma(L)) =
\frac{1}{2}\left(1+\frac{1}{1+2\lambda}\right) \right\}.
\end{eqnarray}
Then
\begin{eqnarray}
\nonumber \frac{\p^0_T(\sigma(L) \in \mathcal G)}{\p^1_T(\sigma(L) \in
    \mathcal G)} =
\frac{\alpha}{1-\alpha}\frac{1+\lambda}{\lambda}.
\end{eqnarray}

\item Let $\beta>\ln 2 - \ln \ln 2$ and $\alpha = \frac1d\big(\ln d + \ln
  \ln d - \ln
  \ln \ln d -\beta \big)$. Then in the 3-level $d$-ary tree $T(3)$ we have that
  $$\E^1_{T(3)}(\p(\sigma_\rho=1 | \sigma(L))) \leq \frac{1}{2}.$$

\end{enumerate}

\end{proposition}

\begin{proof} Part $a)$ is a consequence of standard tree recursions
  for Markov models established using Bayes rule.
For part $b)$ first note that
\begin{eqnarray}\label{e:calARelation}
\p_T(\sigma_x=1 \ | \ \sigma(L) \in \mathcal G) & = & 1-\p_T(\sigma_x=0
  \ | \ \sigma(L) \in \mathcal G) \nonumber\\
& = & \frac{1}{2}\left(1-\frac{1}{1+2\lambda}\right)
\end{eqnarray}
Now,
\begin{eqnarray}
\nonumber \p^0_T(\sigma(L) \in \mathcal G)  & = & \frac{\p_T(\sigma_x=0 \ |
    \ \sigma(L) \in \mathcal G) \p_T( \sigma(L) \in \mathcal G)}{1-\alpha}\\
\nonumber & = &  \frac{\alpha}{1-\alpha} \frac{1+\lambda}{\lambda}
\left(\frac{\p_T(\sigma_x=1 \ |
    \ \sigma(L) \in \mathcal G) \p_T( \sigma(L) \in \mathcal
    G)}{\alpha}\right)\\
\nonumber & = & \frac{\alpha}{1-\alpha}
\frac{1+\lambda}{\lambda}\p^1_T(\sigma(L) \in \mathcal G)
\end{eqnarray}
where the first and third equations follow by definition of
conditional probabilities and the second follows
from \eqref{e:calARelation} and the definition of $\mathcal G$ which establishes $b)$.

For part $c)$, we start  by calculating the probability of certain
posterior probabilities for  trees of small depth. With our
assumption on $\al$ we have that 
\[
\lambda=\frac{\al}{1-2\al}\left(1+\frac{\al}{1-2\al}\right)^{d}=\frac{ (1+o_d(1))e^{-\beta}  \ln^2 d}{ \ln \ln d}
\]
Since $\sigma(L) = 0$ under $\p^1_{T(1)}$, by part $a)$ we have that
\[
\p^1_{T(1)}(\sigma_x=0|\sigma(L)) =  \frac{1}{1+\lambda}  \ w. p. \ 1.
\]

Also,
\[
\p_{T(1)}( \forall \ i, \ u_i=0 | \sigma_x=0) =
\left(\frac{1-2\al}{1-\al}\right)^d 
\]

Using the two equations above, we have that

\[
\p^0_{T(1)}(\sigma_x=0|\sigma(L)) = \left\{
\begin{array}{ll}
1 & \ w.p. \  1- \left(\frac{1-2\al}{1-\al}\right)^d\\
\frac{1}{1+\lambda} & \ w.p. \ \left(\frac{1-2\al}{1-\al}\right)^d.
\end{array}\right.
\]

The first case above corresponds to leaf configurations of the tree
$T(1)$ where at least one of the leaves is 1, while the second case
corresponds to the configurations where all the leaves are 0.
Next, applying part $a)$ to a tree of depth $2$, we have

\[
\p_{T(2)}^1(\sigma_x=0 | \sigma(L)) = \frac{1}{1+\lambda
  \prod_{i}\p_{T(1)}^0(\sigma_{u_i}=0 | \sigma(L))}
\]

Using this expression we can write down this conditional probability
based on the leaf configurations of the subtrees of the root of depth 1.

\begin{eqnarray}\label{eq:p-t-2}
\p_{T(2)}^1(\sigma_x=0 | \sigma(L)) = \left\{
\begin{array}{ll}
\frac{1}{1+\lambda} &
\ w.p.\ \left(1-\left(\frac{1-2\al}{1-\al}\right)^d\right)^d
\\
\frac{1}{2}\left(1+\frac{1}{1+2\lambda}\right) &
\ w.p. \ \left(1-\left(\frac{1-2\al}{1-\al}\right)^d\right)^{d-1}  \left(\frac{1-2\al}{1-\al}\right)^d d
\\
> \frac{1}{2}\left(1+\frac{1}{1+2\lambda}\right) & \ o.w.
\end{array}\right.
\end{eqnarray}

The first case above corresponds to the situation when each subtree of
the root of
depth 1 has a
leaf configuration where at least one of the leaves is 1. The second
case is when one of the $d$ subtrees has a leaf configuration where all
leaves are 0, while the remaining subtrees have leaf configurations
where at least one leaf is 1. The third case corresponds to the
remaining possibilities.

By part $b)$ with $\mathcal G$ as defined, and \eqref{eq:p-t-2} we
have that after substituting the expressions for $\lambda$ and
$\omega$,
\begin{eqnarray}
\p_{T(2)}^0(\sigma(L) \in \mathcal G) & = &
\frac{\alpha}{1-\alpha}\frac{1+\lambda}{\lambda} \p_{T(2)}^1[\sigma(L) \in
  \mathcal G] \nonumber\\
& = & \frac{\al(1+\lambda)}{\lambda(1-\al)}
\left(1-\left(\frac{1-2\al}{1-\al}\right)^d\right)^{d-1}
\left(\frac{1-2\al}{1-\al}\right)^d d\nonumber\\
\label{eq:bound-on-p} & \geq & (1-o_d(1))\frac{e^{\beta} \ln \ln k }{k}
\end{eqnarray}

We can now calculate the values of $P_{T(3)}^1(\sigma_x=0 | \sigma(L))$
as follows. By part $a)$

\[
\p_{T(3)}^1(\sigma_x=0 | \sigma(L))  = \frac{1}{1+\lambda
  \prod_{i}\p_{T(2)}^0(\sigma_{u_i}=0 | \sigma(L))}
\]

Denote
\[
p=\frac{\al(1+\lambda)}{\lambda(1-\al)}
\left(1-\left(\frac{1-2\al}{1-\al}\right)^d\right)^{d-1}
\left(\frac{1-2\al}{1-\al}\right)^d d
\]

Thus, $p$ is the probability that if we started with $\sigma_\rho=0$
in $T(2)$, the configuration at the leaves is from $\mathcal G$. If we
start with $\sigma_\rho=1$ in $T(3)$, the number subtrees of the root
with leaf configurations in $\mathcal G$ is distributed binomially and
will be about $dp$.
By Chernoff bounds, and the bound on $p$ from
\eqref{eq:bound-on-p}, $$\p\left(Bin(d,p)< e^{\beta}\ln \ln d -
2\sqrt{e^{\beta}\ln \ln d}\right)< \frac{1}{3}.$$ Finally, by the definition
of $\mathcal G$,

\[
\p_{T(2)}^0(\sigma_{u_i}=0 | \sigma(L) \in \mathcal G) =
\frac{1}{2}\left(1+\frac{1}{1+2\lambda}\right)
\]

and hence,

\begin{eqnarray*}
\E^1_{T(3)}(\p(\sigma_x=1 | \sigma(L))) & = & \E^1_{T(3)}(1-\p(\sigma_x=0 | \sigma(L)))\\
& \leq & \left(1- \frac{1}{1+\lambda
  (2(1-o_d(1)))^{-(e^\beta \ln \ln d - 2\sqrt{e^\beta \ln \ln d)}}} \right)\frac{2}{3}
+\ \frac{1}{3}
\end{eqnarray*}
By taking $d$ large enough above, we conclude that for $\beta$ as in the assumptions and
large enough $d$,
\[
\E^1_{T(3)}(\p(\sigma_x=1 | \sigma(L))) \leq \frac{1}{2}
\]
\end{proof}

\begin{lemma}\label{lem:finite-levels}Let $\beta>\ln 2 - \ln \ln 2$
  and $\al = \frac1d \big(\ln d + \ln \ln d - \ln
  \ln \ln d -\beta \big)$.  For $d$ large enough,
$$\overline X(3) \leq \frac{\al}{2}.$$
\end{lemma}

\begin{proof}
By part $c)$ of Lemma \ref{lem:basic-relations}, and part $c)$ of
Proposition \ref{prop:expected-small},
\begin{eqnarray*}
\overline X(3) & = & \frac{1}{\pi_{01}^{2}} \left(
\frac{\E^1_{T_3}(\p(\sigma_x=1 \ | \ \sigma(L))) }{\alpha}- 1 \right) \\
& \leq & \frac{1}{\pi_{01}^{2}} \left(
\frac{1}{2\alpha}- 1 \right)\\
& \leq  & \frac{\al}{2}
\end{eqnarray*}
\end{proof}

Next, we present a recursion for $\overline X$ and complete the proof of
the main result. The development of the recursion follows the steps
in \cite{BorChaMosRoc06} closely so we follow their notation and omit some of
the calculations.

%

\begin{figure}[ht]
\begin{minipage}[b]{0.45\linewidth}
\centering
\includegraphics[width=9cm]{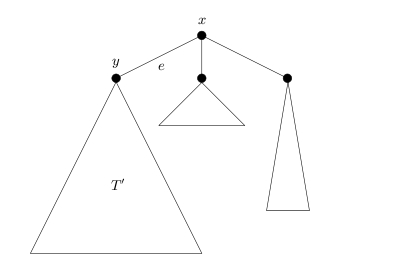}
\caption{A finite tree $T$}
\label{fig:finite-tree}
\end{minipage}
\hspace{0.5cm}
\begin{minipage}[b]{0.45\linewidth}
\centering
\includegraphics[width=10cm]{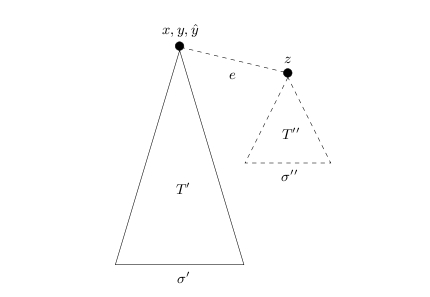}
\caption{The tree $T$ after obtained after merging $T'$ and $T''$. The
dashed subtree is $\hat T$.}
\label{fig:merge-tree}
\end{minipage}
\end{figure}

\subsubsection*{Magnetization of a child} Let $T$ be a finite tree with root $x$ as before. Let $y$ be a child of $x$ and let $T'$ be the subtree of $T$
rooted at $y$ (see Figure \ref{fig:finite-tree}). Let $A'$ be the restriction of $A$ to
the leaves of $T'$. Let $Y=Y(A')$ denote the magnetization of~$y$.

\begin{lemma} \label{lem:child-magnetization} We have
\begin{enumerate}[a)]
\item $\E^1_T(Y) = \theta \E^1_{T'}(Y)$ and $\E^0_T(Y) = \theta
  \E^0_{T'}(Y)$.
\item $\E^1_T(Y^2) = (1-\theta) \E_{T'}(Y^2)+ \theta \E_{T'}^1(Y^2).$
\item $\E^0_T(Y^2) = (1-\theta) \E_{T'}(Y^2)+ \theta \E_{T'}^0(Y^2).$
\end{enumerate}
\end{lemma}

The proof follows from the first part of Lemma \ref{lem:basic-relations} and the
Markov property when we condition on $x$.

Next, we can write the effect on the magnetization of adding an edge
to the root and merging roots of two trees as follows.
Referring to Figure \ref{fig:merge-tree}, let $T'$ (resp. $T''$) be a finite tree
rooted at $y$ (resp. $z$) with the channel on all edges being given
$M$, leaf states $A$ (resp $A''$) and weighted magnetization at the
root $Y$ (resp. $Z$). Now
add an edge $(\hat y, z)$ to $T''$ to obtain a new tree $\hat T$. Then
merge $\hat T$ with $T'$ by
identifying $y = \hat y$ to obtain a new tree $T$. To avoid
ambiguities, denote by $x$ the root of $T$ and $X$ the
magnetization of the root of $T$. We
let $A = (A', A'')$ be the leaf state of $T$. Let $\hat Y$ be the
magnetization of the root of $\hat T$.\\

{\bf Note:} In the above construction, the vertex $y$ is a vertex ``at
the same level'' as $x$, and not a child of $x$ as it was in Lemma
\ref{lem:child-magnetization}. \\

\begin{lemma}\label{lem:add-edge}With the notation above, $\hat Y = \theta Z.$
\end{lemma}

The proof follows by applying Bayes rule, the Markov property and
Lemma \ref{lem:basic-relations}. These facts also imply that

\begin{lemma}\label{lem:merge}
For any tree $\hat{T}$,
\[
X  =  \frac{Y +\hat Y + \Delta Y \hat Y }{1+\pi_{01}Y \hat Y}.
\]
\end{lemma}

With these lemmas in hand we can use derive
a recursive upper bound on the second moments. We will use the
expansion

\[
\frac{1}{1+r} = 1-r + r^2 \frac{1}{1+r}.
\]

Taking $r = \pi_{01}Y \hat Y$, by Lemma \ref{lem:merge} we have
\begin{eqnarray}\label{eq:rec-expansion}
\nonumber X & = & (Y +\hat Y + \Delta Y \hat Y ) \left( 1-\pi_{01}Y \hat Y +
  (\pi_{01}Y \hat Y )^2 \frac{1}{1+\pi_{01}Y \hat Y} \right)\\
\nonumber & = & Y +\hat Y  + \Delta Y \hat Y - \pi_{01}Y \hat Y \left( Y +\hat Y
  + \Delta Y \hat Y \right) + (\pi_{01})^2(Y \hat Y)^2 X\\
& \le & Y +\hat Y  + \Delta Y \hat Y - \pi_{01}Y \hat Y \left( Y +\hat Y
  + \Delta Y \hat Y \right) + (\pi_{01})^2(Y \hat Y)^2
\end{eqnarray}
where the last inequality follows since $X \leq 1$ with probability 1.

Let $\rho' = \overline Y_1 / \overline Y $ and $\rho'' =
\overline Z_1 / \overline Z$. Below, the moments $\overline Y$
etc. are defined according to the appropriate measures over the tree
rooted at $y$ (i.e. $T'$) etc.

Applying Lemmas \ref{lem:basic-relations}, \ref{lem:child-magnetization} and
\ref{lem:add-edge}, we have the following relations.

\begin{eqnarray}
\E^1_T(X) = \pi_{01} \overline X, \ \ \ \E^1_T(Y) = \pi_{01} \overline
y, \ \ \ \E^1_T(Y^2) = \overline Y \rho' \nonumber \\
\label{eq:conditional-second-moment} \E^1_T(\hat Y) = \pi_{01}\theta^2
\overline Z, \ \ \ \E^1_T(\hat Y^2)
= \theta^2 \overline Z ((1-\theta) + \theta \rho'').
\end{eqnarray}


Applying $(\pi_{01})^{-1}E_T^1(\cdot)$ to both sides of
\eqref{eq:rec-expansion}, we obtain the following.
\begin{eqnarray}
\nonumber \overline X & \leq & \overline Y + \theta^2 \overline Z +
\Delta\pi_{01}\overline Y \overline Z
-\pi_{01}\theta^2 \overline Y \overline Z \rho'
-\pi_{01}\theta^2 \overline Y \overline Z ((1-\theta) + \theta \rho'')\\
&& -\Delta\theta^2 \overline Y \overline Z \rho' ((1-\theta) + \theta \rho'')
+ \pi_{01}\theta^2 \overline Y \overline Z \rho' ((1-\theta) + \theta
\rho'')\nonumber \\
& = & \overline Y + \theta^2 \overline Z -\pi_{01} \theta^2 \overline Y \overline Z
(\mathcal A - \Delta \mathcal B) \nonumber
\end{eqnarray}
where
\begin{eqnarray}
\mathcal A & = & \rho' + (1-\rho')((1-\theta) + \theta\rho''),
\nonumber \\
\mathrm{and} \ \
\mathcal B & = & 1- (\pi_{01})^{-1}\rho'((1-\theta) + \theta\rho'') = 1-
\frac{\al}{1-2\al} \rho'((1-\theta) + \theta\rho'') . \nonumber
\end{eqnarray}

If $\mathcal A - \Delta \mathcal B \geq 0$, this would already give a
sufficiently good recursion to show that $\overline X(n)$ goes to $0$,
so we will assume is negative and try to get a good (negative) lower bound.
First note that by their definition $\rho',\rho'' \geq 0$. Further
since $\overline Y = \alpha \overline Y_1 + (1-\alpha) \overline Y_0$,  \[
\rho' \leq \frac 1\al.\] Similarly, \[ \rho''
\leq  \frac 1\al.\]

Since $\E_T^1(\hat Y^2)$ and $\overline Z \geq 0$, it follows from
\eqref{eq:conditional-second-moment} that $(1-\theta) + \theta\rho'' \geq
0$. Together with the fact that $\rho' \geq 0$, this implies that
$\mathcal B \leq 1$.

Since $\mathcal A$ is multi-linear in $(\rho',\rho'')$, to minimize it,
its sufficient to consider the extreme cases. When $\rho' = 0$,
$\mathcal A$ is minimized at the upper bound of $\rho''$ and hence
\[
\mathcal A \geq 1-\pi_{01}\frac{\al}{1-\al} = 0.
\]

When $\rho' = \frac{1}{\alpha}$,
\[
\mathcal A = \frac{1}{\alpha} + \left(1-\frac{1}{\alpha}\right)[1-\theta(1-\rho'')] \geq 0.
\]

Hence, we have
\[
\overline X \leq \overline Y + \theta^2 \overline Z + \frac{1-2\al}{1-\al}\overline Y \overline Z.
\]

Applying this recursively to the tree, we obtain the following
recursion for the moments.

\[
\overline X \leq \frac{1-\al}{1-2\al}\theta^2 \left(\left(1+Z \frac{1-2\al}{1-\al}\right)^k-1 \right)
\]

We bound the $(1+x)^k-1$ term as,
\[
|(1+x)^k -1| \leq e^{|x|k}-1 =\int_0^{|x|k} e^s \ ds \leq e^{|x|k}k|x|
\]

and this implies the following recursion.

\begin{proposition}\label{prop:recursion}
If for some $n$, $\overline X(n) \leq \frac{\al}{2}$, we have that
\[
\overline X(n+1) \leq \theta^2 \left(\frac{1-\al}{1-2\al} \right)^2 e^{\frac{1}{2}\al d}d \overline X(n).
\]
\end{proposition}

Thus if $ \left(\frac{\al}{1-2\al} \right)^2 e^{\frac{1}{2}\al d}d < 1$ then it follows from the recursion that
\begin{equation}\label{e:recursionLimit}
\lim_n \overline X(n) = 0.
\end{equation}
When $\al = \frac1d \big(\ln d + \ln \ln d - \ln \ln \ln d
  -\beta\big)$ and $\beta> \ln 2 - \ln \ln 2$, by Lemma
\ref{lem:finite-levels}, for $d$ large enough, $\overline X(3) \leq
\frac{\al}{2}$. Hence by equation \eqref{e:recursionLimit} we have that
$\overline X(n) \to 0$ and so by Proposition~\ref{p:bar-x-non-reconst}
we have non-reconstruction. Since reconstruction is monotone in
$\lambda$ and hence in $\al$ it follows that we have
non-reconstruction for $\al \leq \al_R$ for large enough $d$.  This
completes the proof of Theorem~\ref{thm:tree-non-reconst}.

\section{Partition function of the hardcore model for random $d$-regular graph}\label{sec:partition-fn-G}
In this section, we derive expressions for the first and second moments of the hardcore partition function for the $d$-regular random graph. The  calculations are along the lines of those in \cite{MosWeiWor09} and we adopt their notation here. We will work with the configuration model for random graphs, described below, in order to simplify the calculations.

\subsection{Configuration model}
Let $\mathcal H(n,d)$ denote the set of all $d$-regular (multi)graphs on $n$ vertices and $\mathcal G(n,d)$ the subset of $d$-regular simple graphs.  The analysis of the properties of a random graph in $\mathcal G(n,d)$ can often be simplified by making use of the {\em configuration model}, introduced by Boll\'{o}bas \cite{Bol80}. Fix $d$ and $n$ such that $dn$ is even. Define a $d$-regular multigraph on $n$ vertices via the configuration model as follows. Begin by replacing each vertex with $d$ distinct copies and then generate a uniformly random pairing of the $dn$ distinct points. Finally, collapse the $d$ copies corresponding to each vertex back into one vertex, obtaining a uniformly random multigraph in $\mathcal H(n,d)$. Let $S$ be the event that the multigraph obtained is simple. Clearly, on the event $S$, the graph obtained is uniformly distributed over $\mathcal G(n,d)$. Moreover, for fixed $d$,
\begin{equation}\label{eq:prob-simple-config}
\P(S) = (1+o(1)) \exp \left( \frac{1-d^2}{4}\right)\,,
\end{equation}
where the $o(1)$ term tends to $0$ as $n \to \infty$. Since the probability in \eqref{eq:prob-simple-config} is uniformly bounded away from $0$, any event that holds asymptotically with high probability for $\mathcal H(n,d)$ also holds asymptotically with high probability when we condition on $S$, i.e. for $\mathcal G(n,d)$. In what follows, by ``$d$-regular random graph'', we will mean the multigraph generated by the configuration model, unless mentioned otherwise.

One useful property of the configuration model that we will make use of repeatedly is that the pairings of the $dn$ distinct points may be revealed sequentially. That is, given a vertex $v$, we may reveal the pairings of its $d$ copies one by one so that the distribution of pairings over the remaining unmatched points remains uniform.

{\bf Notation:} In the sequel, we will use $f(n) = \widetilde \Theta(g(n))$ to mean equality of the functions up to polynomial factors in $n$. We will assume throughout that quantities of the form $an, \alpha n,\gamma n, \varepsilon n$ are integers. We use ``with high probability" to mean with probability going to 1 as $n \to \infty$. In what follows, we will use $\sigma_u$ to denote the restriction of an independent set $\sigma$ of the graph to a vertex $u$. The restriction of $\sigma$ to a subset of vertices $S$ will be denoted by $\sigma(S)$.

\subsection{The first moment of the partition function}
In this section, we calculate the first moment of the partition function for the hardcore model on the $d$-regular random graph. For an independent set $\sigma \in I(G)$, let $|\sigma|$ denote the number of vertices in $I$. For fugacity $\lambda$, the partition function is given by 
\[
Z_G=Z_G(\lambda)=\displaystyle\sum_{\sigma \in I(G)} \lambda^{|\sigma|}.
\]
Let $0 \le \alpha \leq 1/2$ and let $Z_{G,\alpha}=Z_{G,\alpha}(\lambda)$
be the contribution to the partition function from independent
sets of size $\alpha n$, i.e.
\[
Z_{G,\alpha} := \sum_{\sigma \in I(G):|\sigma| = \alpha n}\lambda^{\alpha n}, \qquad Z_G = \sum_{\alpha} Z_{G,\alpha}.
\]

The following approximation will be useful in simplifying the probabilities obtained in the sequel. Let $a>0$ be a constant. Then, by Stirling's approximation,

\begin{align}\label{lem:product-f-comparison}
\displaystyle\prod_{j=1}^{an} j =  \exp \left( n \int_{0}^{a} \ln(x) dx +a n \ln(n)+ O(\ln n) \right).
\end{align}

Let
\[
H(x) = -x\ln(x)-(1-x)\ln(1-x).
\]
\begin{lemma}\label{prop:Z-alpha-first-moment}
Let $G \sim \mathcal H(n,d)$. Fix $\lambda>0$ and $0 \le \alpha \le \frac 12$. The first moment of $Z_{G,\alpha}$ is given by
\begin{align}\label{eq:first-moment-product}
\E\left(Z_{G,\alpha}\right) = {n \choose \alpha n} \lambda^{\alpha n}
\displaystyle\prod_{i=0}^{\alpha nd-1} \frac{(1-\alpha)nd-i}{nd-1
  -2i}  = \widetilde \Theta(1) \exp\left(n \Phi(\alpha)\right)
\end{align}
where
\begin{align}
\Phi(\alpha) = \Phi(\alpha,\lambda) =H(\alpha) + \alpha \ln(\lambda) + d \left(
(1-\alpha)\ln(1-\alpha) - \left(\frac{1-2\alpha}{2}\right)\ln(1-2\alpha) \right).
\label{eq:Phi1}
\end{align}
\end{lemma}
\begin{proof}
The first equality follows by calculating the probability in the configuration model that a given subset of $\alpha n$ vertices is an independent set, i.e. that the vertices in the subset are not matched to vertices in the subset itself.
The second equality follows by \eqref{lem:product-f-comparison}.
\end{proof}
For $\lambda>0$, it can be verified that the maximum of $\Phi$ is achieved at  $\alpha^* = \alpha^*(\lambda,d)$, which is the solution to the equation
\begin{align} \label{eq:alpha-star-eqn}
\lambda
\frac{1-\alpha}{\alpha}\left(\frac{1-2\alpha}{1-\alpha}\right)^{d}
= 1\,
\end{align}
which is obtained by differentiating $\Phi$. To solve, if we were to set $\alpha=x/d$, this would reduce roughly to
solving $xe^x=\lambda d$ and thus we obtain that
\[
\alpha^*(\lambda,d)= (1+o(1))\frac{\ln(\lambda d)}{d}.
\]
Notice that the relation \eqref{eq:alpha-star-eqn} between $\alpha, \lambda$ and $d$ is equivalent to~\eqref{e:alphaLambdaRelation}.

\subsection{The second moment of the partition function}\label{sec:second-moment-overview}
To estimate the second moment $\E((Z_{G,\alpha})^2)$, we consider the contributions from pairs of independent sets $S,T$ each of size $\alpha n$. We divide this according to the size of the
overlap $|S \cap T|=\gamma n$ and according to the number $\varepsilon n$ of edges of the graph which go from each
of $S,T$ to the complement $V \setminus (S \cup T)$. Call this
contribution $Z^{(2)}_{G,\alpha,\gamma,\varepsilon}$. That is, for $(\alpha,\gamma,\varepsilon)$ in the region
\begin{align}
\mathcal R = \left\{ (\alpha,\gamma,\varepsilon): 0 \le \alpha,\gamma,\varepsilon \le \frac 12, \ \ \alpha - \gamma -\varepsilon \ge 0,
\ \ 1-2\alpha-2\varepsilon \ge 0\right\}, \label{eq:feasible-region}
\end{align}
we define
\begin{align*}
& Z^{(2)}_{G,\alpha,\gamma,\varepsilon}  \defeq \\ 
& \lambda^{2 \alpha n}  \Big | \Big \{S,T \in I(G): |S|=|T|=\alpha n, |S \cap T|=\gamma n, |E_G(S,V \setminus (S \cup T))| = |E_G(T,V \setminus (S \cup T))| =\varepsilon n  \Big \} \Big |
\end{align*}
and
\[
\E((Z_{G,\alpha})^2) = \displaystyle\sum_{\gamma,\varepsilon}\E\left( Z^{(2)}_{G,\alpha,\gamma,\varepsilon} \right).
\]
Calculating the probability in the configuration model that a pair of subsets of vertices $S$ and $T$ as above are both independent sets, we have that for $G \sim \mathcal H(n,d)$,
\begin{align}\label{eq:second-moment-product}
\E\left(Z^{(2)}_{G,\alpha,\gamma, \varepsilon}\right) = \lambda^{2 \alpha n}{n\choose
  \alpha n} {\alpha n \choose \gamma n}
{(1-\alpha)n \choose (\alpha-\gamma)n}
 \left(
\frac{\displaystyle\prod_{i=0}^{\gamma n d-1}
  (1-2\alpha+\gamma)dn - i}{\displaystyle\prod_{i=0}^{\gamma n d -1} dn -1 - 2i}\right) \times
\nonumber \\
\times \left(
\frac{\displaystyle\prod_{i=0}^{\varepsilon d n-1}  ((1-2\alpha)dn - i)
  \cdot \displaystyle\prod_{i=0}^{(\alpha-\gamma-\varepsilon)dn -1}
  (\alpha-\gamma)dn - i}{\displaystyle\prod_{i=0}^{(\alpha-\gamma)d n-1}(
  (1-2\gamma)dn -1 -2i)}
\cdot  \frac{\displaystyle\prod_{i=0}^{\varepsilon d n-1}
  (1-2\alpha-\varepsilon)dn - i
}{\displaystyle\prod_{i=0}^{\varepsilon d n-1} (1-2\alpha)dn -1 - 2i}
\right).
\end{align}

The following function arises in the approximation of the expression in \eqref{eq:second-moment-product}
\[
f(\alpha,\gamma,\varepsilon) = 2\alpha
\ln(\lambda)+H(\alpha)+H_1(\gamma,\alpha)+ H_1(\alpha-\gamma,1-\alpha)
+d\Psi_2(\alpha,\gamma,\varepsilon)
\]
where
\[
H_1(x,y) = -x(\ln(x)-\ln(y))+(x-y)(\ln(y-x)-\ln(y))
\]
and
\begin{align*}
& \Psi_2(\alpha,\gamma,\varepsilon) = H_1(\varepsilon,\alpha-\gamma)+
\int_0^\gamma
\ln(1-2\alpha+\gamma-x)d{x} - \int_0^\gamma \ln(1-2x)d{x} \nonumber \\
& + \int_0^{\varepsilon}
\ln(1-2\alpha-x)d{x} +
\int_0^{\alpha-\gamma-\varepsilon}\ln(\alpha-\gamma-x)d{x} -
\int_0^{\alpha-\gamma} \ln(1-2\gamma-2x)d{x} \nonumber \\
& +  \int_0^\varepsilon
\ln(1-2\alpha-\varepsilon-x)d{x} - \int_0^\varepsilon
\ln(1-2\alpha-2x)d{x}. \label{eq:Psi2-integral}
\end{align*}

In particular, in Section \ref{sec:tech-lemmas} we will show that the logarithm of $\E(Z^{(2)}_{G,\alpha,\gamma,\varepsilon})$ scaled by $n$ is well approximated by $f$, and for $\lambda < \lambda_c$, $f$ decays quadratically around its maximum.

\begin{proposition}\label{prop:f-approximates-Z-square}
Let $G \sim \mathcal H(n,d)$ and $0 \le \alpha \le \frac 12$. Then,
\begin{align*}
\E\left(Z^{(2)}_{G,\alpha,\gamma,\varepsilon}\right)= \exp\left(nf(\alpha,\gamma,\varepsilon)+ O(\ln(n))\right).
\end{align*}
\end{proposition}

For any $\alpha$, define $\hat \gamma = \hat \gamma(\alpha) \defeq \alpha^2$ and $\hat \varepsilon = \hat \varepsilon(\alpha) \defeq \alpha(1-2\alpha)$. We will also use the shorthand $\gamma^*\defeq\hat \gamma(\alpha^*)$ and $\varepsilon^* \defeq \hat \varepsilon(\alpha^*)$.

\begin{proposition}\label{prop:f-global-max}
Let $\lambda < \lambda_c$, and $(\alpha,\gamma,\varepsilon) \in \mathcal R$. Then, the function $f(\alpha,\gamma,\varepsilon)$ attains its maximum at $(\alpha^*,\gamma^*,\varepsilon^*)$ and is strictly concave at this point. In particular for some $C=C(d,\lambda)$,
\begin{align*}
f(\alpha^*,\gamma^*,\varepsilon^*) - f(\alpha,\gamma,\varepsilon)  \ge  C (|\alpha-\alpha^*|^2 + |\gamma-\gamma^*|^2+ |\varepsilon-\varepsilon^*|^2).
\end{align*}
\end{proposition}

%
%

Finally, the following second moment-type bound is also proved in Section \ref{sec:tech-lemmas}.
\begin{proposition}\label{prop:second-moment-square-of-first}Let $G \sim \mathcal H(n,d)$ and $\lambda < \lambda_c$. Then,
\begin{align*}
\E((Z_G)^2) =  \widetilde \Theta(1) (\E(Z_G))^2 .
\end{align*}
\end{proposition}

\section{Partition function of a punctured random graph}\label{sec:tildeG}
In this section we study the effect on the hardcore measure of a $d$-regular random graph of conditioning on the spins of a small number of the vertices. In order to do this, we analyze the partition function of a punctured graph
$\tilde{G}$ obtained from a $d$-regular random graph $G$ by deleting a small fraction of vertices and their
neighborhoods.
Define the following
quantities with respect to a graph $G=(V,E)$. Let
$d(u,v)= d_G(u,v)$ denote the distance between two vertices $u,v \in
V$. For a vertex $u \in V$ and integer $r$, the {\em $r$-neighborhood
  of $u$}, denoted $B_r(u)$ and its (vertex) boundary are defined as
\[
B_r(u) := \{v \in V : d(u,v) \le r\}, \ \partial B_r(u) := B_r(u)
\setminus B_{r-1}(u).
\]

%
%
%

\begin{lemma}\label{lem:close-centers-unlikely}
Let $G=(V,E) \sim \mathcal H(n,d)$. Let $S \subset V$ be a set of
vertices with $|S| = n^{3/5}$ and let  $r$ be some large constant. Then, the expected number of $u \in S$
such that the neighborhood $B_r(u)$ contains another vertex in $S$ is
$O(n^{1/5})$ and the probability of a neighbourhood with 3 vertices in $S$ is  $O(n^{-1/5})$.  Furthermore, with high probability, for all $u\in S$ the neighbourhood  $B_r(u)$ is  a tree.
\end{lemma}
\begin{proof}
The bounds on the number of vertices of $S$ in a local neighbourhood follow from the independence of the set $S$ and the graph $G$ and a union bound.  The number of cycles of length at most $2r$ has constant expected value which implies that with high probability the neighbourhoods of vertices in $S$ are trees.
\end{proof}

Let $G=(V,E) \sim \mathcal H(n,d)$ and fix a large constant $r$. Let $S \subset V$ be a uniformly chosen set of
vertices with $|S| = n^{3/5}$.
Let $\tilde G=(\tilde
V,\tilde E)$ be the graph obtained by deleting from $V$ the set of
vertices  $\displaystyle\cup_{v \in S}B_{r-1}(v)$ and any edges adjacent to
these vertices.
Define
\begin{align}\label{eq:def-of-B}
B:= \tilde G \bigcap \left(\bigcup_{u \in S} \partial B_r(u)\right).
\end{align}
Let
\begin{align}\label{eq:centers-S}
S' = \{s \in S \ : \ \forall s' \in S \setminus \{s\}, B_{r}(s) \cap B_{r}(s')= \varnothing, B_r(s) \text{ is a tree} \}
\end{align} 
Let $k = |S'|$. Let $s_1,\ldots,s_k$ be an arbitrary ordering of the elements of $S'$ and for $1 \le i \le k$ define $W_i := \partial B_r (s_i)$. Define $W_{k+1}:=B \setminus \cup_{i=1}^k W_i$.

\begin{corollary}\label{cor:Gtilde-degrees}
The vertices of $B$ have degree $d-1$ or $d-2$ in $\tilde G$ with high probability. With high probability, the number of vertices in $\tilde G$
of degree $d-2$ is $O(n^{\frac 15})$ and the number of vertices of degree
$d-1$ is $n^{\frac 35}(1-o(1))d(d-1)^r$. The size of  $S'$, $k=(1-o(1))n^{\frac 35}$.
\end{corollary}
\begin{proof}
Suppose a vertex $v \in B$ is in $\partial B_r(u_1) \cap
\partial B_r(u_2)$ for
some $u_1,u_2$. We know that with high probability it is not in any third $\partial B_r(u_3)$, otherwise
there are 3 centers close together contradicting Lemma
\ref{lem:close-centers-unlikely}. Therefore its degree in $\tilde G$ is at least
$d-2$ since  there are at most two vertices adjacent to it in
$\displaystyle\cup_{v \in S}B_{r-1}(v)$. In the other case, $v \in \partial B_r(u)$ for a unique
vertex $u \in S$ and hence its degree in $\tilde G$ is $d-1$. The bounds on the numbers of these vertices follow by Lemma \ref{lem:close-centers-unlikely} and applying the second moment method.
The bound on the size of $k$ follows immediately from Lemma \ref{lem:close-centers-unlikely}.
\end{proof}

In what follows we will sometimes work in the conditional space of the configuration model where $G$ is such that the conclusions of Lemma \ref{lem:close-centers-unlikely} and Corollary \ref{cor:Gtilde-degrees} hold for $\tilde G$. Since the configuration model allows us to expose edges and maintain the uniform distribution over pairings of the unmatched pairs, under the conditioning, $\tilde G$ is a graph chosen according to the configuration model where the degrees of the vertices are modified appropriately, and we denote this set of graphs by $\hat{ \mathcal H}(n,d)$. We use $\hat{\mathbb P}$ and $\hat{\mathbb E}$ to denote the corresponding conditional expectation and probabilities.

%

\subsection{The First Moment of the Partition Function of $\tilde G$}
Let $B$ be the subset of vertices defined in \eqref{eq:def-of-B} and let $\sigma \in \{0,1\}^B$. Define $Z_{G,\sigma}$ to be the partition
function over independent sets of $G$ whose restriction to $B$ is $\sigma$,
i.e.,
\[
Z_{G,\sigma} \defeq \displaystyle\sum_{\omega \in I(G) :
  \omega(A)=\sigma} \lambda^{|\omega|}.
\]
Similarly, define
\[
Z_{\tilde G,\sigma} \defeq \displaystyle\sum_{\omega \in I(\tilde G) :
  \omega(A)=\sigma} \lambda^{|\omega|}.
\]

In this section, we will show that in expectation, for the boundary $B$ as defined in \eqref{eq:def-of-B} and any $\sigma \in \{0,1\}^B$, $Z_{\tilde{G},\sigma}$ is
essentially proportional to a product measure on $B$.
Let $m=|V(\tilde G) \setminus B|$.
Define $Z_{\tilde G,\alpha,\sigma}$ to be the partition function for independent
sets of $\tilde G$ whose restriction to $B$ is $\sigma \in \{0,1\}^B$  and for which
$\alpha$ fraction of the vertices $V(\tilde G) \setminus B$ are in the
independent set
\[
Z_{\tilde G,\alpha,\sigma} \defeq \displaystyle\sum_{\omega \in I(\tilde G) :
  \omega(B)=\sigma, \sum_{v \in V \setminus B}\omega_v = \alpha m}
\lambda^{|\omega|}.
\]
Fix an independent set $\omega$ of $\tilde G$ whose restriction to $B$ is $\sigma$ such that $\sum_{v \in V(\tilde G) \setminus B} \omega_v = m$.
Let $L=|\sigma|$ and let $L_i$ be the number of vertices in $\sigma$ of degree $d-i$ for $i=1,2$. Let $M_i$ denote the number of vertices of $B$ of degree $d-i$ for $i=1,2$.  We can calculate the expectation of the partition function as before using the exploration process in the configuration model. Let $N_1 = (d-1)L_1+ (d-2)L_2 + d\alpha m$ be the number of half-edges adjacent to a vertex in the independent set.
 Let $N_T=(d-1)M_1 + (d-2)M_2 + d m$ be the total number of
half-edges overall.
Calculating the probability that the pairing of the half edges does not pair vertices which are in the independent set, we have 
\begin{align}
\label{eq:fraction}
\hat \E\left(Z_{\tilde G, \alpha,\sigma}\right) & = \lambda^{L_1+L_2+\alpha m}
  {m \choose \alpha m}
  \frac{\displaystyle\prod_{i=0}^{N_1-1}(N_T - N_1 - i )}{\displaystyle\prod_{i=0}^{N_1-1}(N_T-1-2i)
}.
\end{align}

In what follows, let $G \sim \mathcal H(n,d)$ and let $\tilde G$ be defined as above. Recall that $\alpha^*$ is given by the solution to \eqref{eq:alpha-star-eqn}. Let $\sigma_0$ denote the empty independent set on $B$.

\begin{proposition}\label{prop:ratio-of-first-moments}
Fix $\lambda > 0$. 
\begin{enumerate}[(1)]
\item \label{proppart:general-alpha} For all $\sigma \in \{0,1\}^B$, and $0 \le \alpha \le \frac 12$,
\begin{align*}
\frac{\hat \E\left(Z_{\tilde G,\alpha, \sigma}\right)}{\hat \E\left(Z_{\tilde G,\alpha,{\sigma_0}} \right)}= \exp\left(O(n^{\frac 15})\right)
\left( \lambda
\left(\frac{1-2\alpha}{1-\alpha}\right)^{d-1} \right)^{|\sigma_B|}.
\end{align*}

\item \label{proppart:firstmom-alpha-close-to-alpha*} Let $\alpha$ be such that $|\alpha-\alpha^*| < Cn^{-\frac 25}$ for a constant $C>0$. Then,
\begin{align*}
\frac{\hat \E\left(Z_{\tilde G,\alpha, \sigma}\right)}{\hat \E\left(Z_{\tilde G,\alpha,{\sigma_0}} \right)}= \exp\left(O(n^{\frac 15})\right)\chi(\sigma)
\end{align*}
where  
\begin{align*}
\chi(\sigma)=\left( \lambda
\left(\frac{1-2\alpha^*}{1-\alpha^*}\right)^{d-1} \right)^{|\sigma|}.
\end{align*}
\end{enumerate}
\end{proposition}

\begin{proof}
We compare the formula \eqref{eq:fraction} for $\sigma$ and ${\sigma_0}$. Let $N_1', N_T'$ be the corresponding quantities for $\sigma_0$ as defined before. Note that for $i=1,2$, $L_i' = 0$ and $N_T' = N_T$. Comparing the numerators and denominators of the fraction in \eqref{eq:fraction} we obtain that
\begin{align*}
\frac{\hat \E(Z_{\tilde G,
  \alpha,\sigma})}{\hat \E(Z_{\tilde G,
  \alpha,{\sigma_0}})} & = \lambda^{L_1+L_2}\frac{\displaystyle\prod_{i=0}^{N_1-N_1'-1} (N_T-2N_1' -2i)
}{\displaystyle\prod_{i=0}^{N_1-N_1'-1}(N_T-N_1'-i)}
 = \lambda^{L_1+L_2} \left(\frac{1-2\alpha}{1-\alpha}\right)^{N_1-N_1'} \times \\
& \times \frac{\displaystyle\prod_{i=0}^{N_1-N_1'-1} \left(1+ \frac{1}{(1-2\alpha) d m}
  \left((d-1)M_1 + (d-2)M_2-2i \right)\right)
}{\displaystyle\prod_{i=0}^{N_1-N_1'-1}\left(1+ \frac{1}{(1-\alpha) d m}
  \left((d-1)M_1 + (d-2)M_2 -i \right)\right)   } \\
& = \lambda^{L_1+L_2} \left(\frac{1-2\alpha}{1-\alpha}\right)^{N_1-N_1'} \exp\left(O\left(\frac{(N_1-N_1') n^{\frac{3}{5}})}{dm}\right)\right),
\end{align*}
where the last line follows since $M_1 \le O(n^{\frac 35})$ and $M_2 \le O(n^{\frac 15})$.
Since $N_1 - N_1' = (d-1)L_1 + (d-2)L_2 \le O(n^{\frac 35})$ and $m = n(1-o(1))$, we obtain that
\begin{align*}
&  \frac{\hat \E\left(Z_{\tilde G, \alpha,\sigma}\right) }{\hat \E\left(Z_{\tilde G, \alpha,{\sigma_0}}\right) } =  \lambda^{L_1+L_2}
\exp\left(O(n^{\frac 15})\right)
\left(\frac{1-2\alpha}{1-\alpha}\right)^{(d-1)L_1+(d-2)L_2}  \\
& =  \exp\left(O(n^{\frac 15})\right)\left(\lambda \left(\frac{1-2\alpha}{1-\alpha}\right)^{d-1}\right)^{|\sigma|} \left(\frac{1-2\alpha}{1-\alpha}\right)^{- L_2} = \exp\left(O(n^{\frac 15})\right) \left(\lambda \left(\frac{1-2\alpha}{1-\alpha}\right)^{d-1}\right)^{|\sigma|} .
\end{align*}
The last bound follows since $L_2 \le O(n^{\frac 15})$ by 
the assumed conditioning, giving part \eqref{proppart:general-alpha} of the proposition. 
Finally, by the assumption that $|\alpha-\alpha^*| \le Cn^{-\frac{2}{5}}$ the last expression above can be bounded by
\begin{align*}
& = \exp\left(O(n^{\frac 15})\right) \left(\lambda \left(\frac{1-2\alpha^*}{1-\alpha^*}\right)^{d-1}\right)^{|\sigma|},
\end{align*}
completing part \eqref{proppart:firstmom-alpha-close-to-alpha*} of the proposition.
\end{proof}

\begin{proposition}\label{prop:tildeG-first-mom-alpha-star}
For all $\sigma \in \{0,1\}^B$, and for large enough constant $C = C(\lambda,d)$,
\begin{align*}
\hat \E(Z_{\tilde G,\sigma}) = (1-o(1)) \sum_{\alpha: |\alpha - \alpha^*| \le  Cn^{-\frac 25}} \hat \E(Z_{\tilde G,\alpha,\sigma}).
\end{align*}
\end{proposition}

\begin{corollary}\label{cor:tildeG-sigmaB-sigma0-firstmom}
For any $\sigma \in \{0,1\}^B$,
\[
\hat \E(Z_{\tilde G,\sigma}) = \exp\left( O(n^{\frac 15}) \right) \chi(\sigma_B) \hat \E(Z_{\tilde G,\sigma_0}).
\]
\end{corollary}
\begin{proof}
The claim follows by putting together part \eqref{proppart:firstmom-alpha-close-to-alpha*} of Proposition \ref{prop:ratio-of-first-moments} and Proposition \ref{prop:tildeG-first-mom-alpha-star}.
In particular, taking $C$ to be large enough as in Proposition \ref{prop:tildeG-first-mom-alpha-star},
\begin{align*}
\hat \E(Z_{\tilde G,\sigma}) = \tilde \Theta\left(  1 \right) \sum_{\alpha: |\alpha - \alpha^*| \le  Cn^{-\frac 25}} \hat \E(Z_{\tilde G,\alpha,\sigma}) & = 
 \exp\left( O(n^{\frac 15}) \right) \chi(\sigma) \sum_{\alpha: |\alpha - \alpha^*| \le  Cn^{-\frac 25}}  \hat \E(Z_{\tilde G,\sigma_0}) \\ & = \exp\left( O(n^{\frac 15}) \right) \chi(\sigma) \hat \E(Z_{\tilde G,\sigma_0}).
\end{align*}
\end{proof}

To prove Proposition \ref{prop:tildeG-first-mom-alpha-star}, we need a few intermediate results. Let $G_m \sim \mathcal H(m,d)$ where $m = |V(\tilde G) \setminus B|$, as defined above. Define the partition functions $Z_{G_m,\alpha}$ and $Z^{(2)}_{G_m,\alpha,\gamma,\varepsilon}$ respectively as $Z_{G,\alpha}$ and $Z^{(2)}_{G,\alpha,\gamma,\varepsilon}$ were defined, with $G=G_m$.

\begin{lemma}
\label{lem:tildeG-firstmom-2}
For any $0 \le \alpha \le \frac 12$,
\begin{align*}
\frac{\hat \E(Z_{\tilde G,\alpha,\sigma_0})}{ \hat \E(Z_{\tilde G,\alpha^*,\sigma_0})} =  \frac{\E(Z_{G_m,\alpha})}{\E(Z_{G_m,\alpha^*})}
\exp \left( O(n^{\frac 35} |\alpha - \alpha^*| )\right).
\end{align*}
\end{lemma}

\begin{proof}
Let $N_1$ and $N_T$ be defined as in \eqref{eq:fraction}. Define $N_1^*$ and $N_T^*$ analogously for $\alpha^*$. Note that for the configuration $\sigma_0$, $L_i =0$ for $i=1,2$ and $N_T = N_T^*$.
Comparing the expressions in \eqref{eq:fraction} and \eqref{eq:first-moment-product}, we have that
\begin{align*}
\frac{\hat \E(Z_{\tilde G,\alpha,\sigma_0})}{\hat \E(Z_{\tilde G,\alpha^*,\sigma_0})}  & = \frac{\E(Z_{G_m,\alpha})}{\E(Z_{G_m,\alpha^*})}  \times 
\frac{\displaystyle\prod_{i=0}^{d \alpha m -1} \frac{dm-2i-1}{dm+(d-1)M_1+(d-2)M_2-2i-1}}
{\displaystyle\prod_{i=0}^{d \alpha^* m -1} \frac{dm-2i-1}{dm+(d-1)M_1+(d-2)M_2-2i-1}}
\times\\
& \times \frac{\displaystyle\prod_{i=0}^{d \alpha m -1} \frac{(1-\alpha)dm+(d-1)M_1+(d-2)M_2-i}{(1-\alpha)dm-i}}
{\displaystyle\prod_{i=0}^{d \alpha^* m -1} \frac{(1-\alpha^*)dm+(d-1)M_1+(d-2)M_2-i}{(1-\alpha^*)dm-i}}\\
& = \frac{\E(Z_{G_m,\alpha})}{\E(Z_{G_m,\alpha^*})} \exp \left( O(n^{\frac 35} |\alpha - \alpha^*| )\right)
\end{align*}
where the last line follows by the bounds on $m$ and $M_i$ for $i = 1,2$.
\end{proof}

\begin{lemma}
\label{lem:tildeG-firstmom-3}
For $0 \le \alpha \le \frac 12$, there is a constant $C=C(\lambda,d) >0$ such that
\begin{align*}
\frac{\E(Z_{G_m,\alpha})}{\E(Z_{G_m,\alpha^*})} = \tilde \Theta \left( \exp \left( - C n|\alpha-\alpha^*|^2  \right) \right).
\end{align*}
\end{lemma}

\begin{proof}Recall the expression  \eqref{eq:Phi1} for $\Phi(\alpha)$.
Writing the Taylor expansion for $\Phi(\alpha)$ around $\alpha^*$ and noting that $\Phi'(\alpha^*) = 0$, we have
\begin{align*}
\Phi(\alpha) - \Phi(\alpha^*) & =  \frac{\partial \Phi(\alpha^*)}{\partial \alpha} |\alpha-\alpha^*| + \frac 12 \frac{\partial^2 \Phi(\alpha^*)}{\partial \alpha^2} |\alpha-\alpha^*|^2 + o\left(|\alpha-\alpha^*|^2 \right) \\
& = - \left( \frac{1}{\alpha^*(1-\alpha^*)} + \frac{d}{(1-\alpha^*)(1-2\alpha^*)}\right) |\alpha-\alpha^*|^2 + o\left(|\alpha-\alpha^*|^2 \right)
\end{align*}
and therefore 
\begin{align*}
\frac{\E(Z_{G_m,\alpha})}{\E(Z_{G_m,\alpha^*})}  = \tilde \Theta \left(  \exp \left( m\left( \Phi(\alpha) -\Phi(\alpha^*) \right) \right) \right) = \tilde \Theta \left( \exp \left( - C n|\alpha-\alpha^*|^2  \right) \right).
\end{align*}

\end{proof}

\begin{proof}[Proof of Proposition \ref{prop:tildeG-first-mom-alpha-star}]

Combining part \eqref{proppart:general-alpha} of Proposition \ref{prop:ratio-of-first-moments}, Lemma \ref{lem:tildeG-firstmom-2} and Lemma \ref{lem:tildeG-firstmom-3}, we have that for any $0 \le \alpha \le \frac 12$,
\begin{align*}
\frac{\hat \E(Z_{\tilde G,\alpha,\sigma})}{\hat \E(Z_{\tilde G,\alpha^*,\sigma_B})} & =
\frac{\hat \E(Z_{\tilde G,\alpha,\sigma})}{\hat \E(Z_{\tilde G,\alpha,\sigma_0})}
 \cdot
\frac{\hat \E(Z_{\tilde G,\alpha,\sigma_0})}{\hat \E(Z_{\tilde G,\alpha^*,\sigma_0})} \cdot
 \frac{\hat \E(Z_{\tilde G,\alpha^*,\sigma_0})}{\hat \E(Z_{\tilde G,\alpha^*,\sigma_B})} \\
 & = \exp \left( O\left(n^{\frac 15}- C n|\alpha-\alpha^*|^2 + n^{\frac 35} |\alpha - \alpha^*|\right)\right) 
\end{align*}
When $|\alpha-\alpha^*| > C n^{-\frac 25}$, for a large enough constant $C = C(\lambda,d)$,  the second term in the parenthesis above dominates and right hand side can be made arbitrarily small. Therefore
\begin{align*}
\hat \E(Z_{\tilde G,\sigma}) & =  \sum_{\alpha: |\alpha - \alpha^*| > Cn^{-2/5}} \hat \E(Z_{\tilde G,\alpha,\sigma}) + \sum_{\alpha: |\alpha - \alpha^*| \le  Cn^{-2/5}} \hat \E(Z_{\tilde G,\alpha,\sigma})\\
& = (1-o(1)) \sum_{\alpha: |\alpha - \alpha^*| \le  Cn^{-2/5}} \hat \E(Z_{\tilde G,\alpha,\sigma}).
\end{align*}
\end{proof}

\subsection{The Second Moment of the Partition Function of $\tilde G$}
As before, we divide the second moment $\E((Z_{\tilde G,\alpha,\sigma})^2)$ into the contribution from pairs of independent sets $S$ and $T$ of $\tilde G$ whose restriction to $B$ is $\sigma$,  $\sum_{v \in V \setminus B}S_v = \sum_{v \in V \setminus B}T_v = \alpha m $ and $|(S \cap T)\setminus B| = \gamma m$. We can further divide according the number $\varepsilon d m$ of half-edges which are paired from each of $S$ and $T$ to $V(\tilde G) \setminus (S \cup T )$. Denote this contribution by $Z^{(2)}_{\tilde G, \alpha, \gamma,\varepsilon, \sigma}$. Thus, we can write
\begin{align*}
& \E\left((Z_{\tilde G, \alpha, \sigma})^2\right)  =
\displaystyle\sum_{\gamma,\varepsilon} \E\left(Z^{(2)}_{\tilde G, \alpha, \gamma,\varepsilon, \sigma}\right).
\end{align*}
As before, let $L$ denote $|\sigma|$ with $L_i$ denoting the numbers of vertices of $B$ in the independent set of degrees $d-i$ for $i=1,2$. Define $M_i$ as before and let $K_i = M_i-L_i$. Calculating the probability in the configuration model that a pair of subsets $S$ and $T$ as above are independent sets we obtain
\begin{align}
& \hat \E\left(Z^{(2)}_{\tilde G, \alpha, \gamma,\varepsilon, \sigma}\right)  = \lambda^{2 \alpha m +
  2(L_1+L_2)}{m \choose  \alpha m }
{\alpha m \choose \gamma m}{(1-\alpha)m \choose (\alpha-\gamma)m} \times
\nonumber \\
\times & \left(
\frac{\displaystyle\prod_{i=0}^{\gamma m d + (d-1)L_1 +(d-2)L_2-1}
  (1-2\alpha+\gamma)dm +(d-1)K_1+(d-2)K_2 - i}{\displaystyle\prod_{i=0}^{\gamma m d
    + (d-1)L_1 +(d-2)L_2-1} dm  + (d-1)M_1+(d-2)M_2 -1 - 2i}\right) \times
\nonumber \\
\times & \left(
\frac{\displaystyle\prod_{i=0}^{\varepsilon d
    m -1}  (1-2\alpha)dm+
  (d-1)K_1+(d-2)K_2 - i \cdot \displaystyle\prod_{i=0}^{(\alpha-\gamma-\varepsilon)d 
    m -1}  (\alpha-\gamma)dm -
  i}{\displaystyle\prod_{i=0}^{(\alpha-\gamma)d m -1}  (1-2\gamma)dm +
(d-1)K_1+(d-2)K_2 -1 -2i} \times \right. \nonumber \\
\times  & \left. \frac{\displaystyle\prod_{i=0}^{\varepsilon d m -1}
  (1-2\alpha-\varepsilon)dm + (d-1)K_1+(d-2)K_2 - i
}{\displaystyle\prod_{i=0}^{\varepsilon d m -1} (1-2\alpha)dm + (d-1)K_1+(d-2)K_2 -1 - 2i}
\right). \label{eq:secondmoment-tilde-G}
\end{align}

We will show below that the second moment $\hat \E\left((Z_{G, \sigma})^2\right)$ is roughly the square of the first moment $\hat \E\left(Z_{G, \sigma}\right)$ by an analysis similar to that in \cite[Theorem 6.11]{MosWeiWor09} and
\cite[Lemma 3.5]{Sly10}.

\begin{proposition}\label{prop:ratio-second-moments}Let $(\alpha,\gamma,\varepsilon) \in \mathcal R$. Then, for any $\sigma \in \{0,1\}^B$,
\[
\hat \E\left(Z^{(2)}_{\tilde G, \alpha, \gamma,\varepsilon, \sigma}\right) \le \exp \left( O\left( n^{\frac 15} \right)\right) \hat \E\left(Z^{(2)}_{\tilde G, \alpha^*, \gamma^*,\varepsilon^*, \sigma}\right).
\]
\end{proposition}

To prove Proposition \ref{prop:ratio-second-moments}, we need a few intermediate lemmas.

\begin{lemma}\label{lem:approximation-to-*}Let $(\alpha,\gamma,\varepsilon) \in \mathcal R$.  Then, for some $C = C(\lambda,d)>0$
\begin{align*}
\frac{\hat \E(Z^{(2)}_{\tilde G, \alpha,\gamma,\varepsilon,\sigma_0})}{\hat \E(Z^{(2)}_{\tilde G, \alpha^*,\gamma^*,\varepsilon^*,\sigma_0})}\leq \exp\left(-C n\left( |\alpha-\alpha^*|^2+ |\gamma-\gamma^*|^2 +|\epsilon-\epsilon^*|^2\right)+ O \left(n^{\frac{3}{5}}\left( |\alpha-\alpha^*|+ |\gamma-\gamma^*|+ |\varepsilon-\varepsilon^*|\right) \right) \right).
\end{align*}
\end{lemma}
\begin{proof}For the configuration $\sigma_0$, for $i=1,2$, $L_i=0$ and $K_i=M_i$. 
Comparing the expressions  \eqref{eq:second-moment-product} and \eqref{eq:secondmoment-tilde-G}, we obtain that for $G_m \sim \mathcal H(m,d)$,

\begin{align*}
&  \frac{\hat \E(Z^{(2)}_{\tilde G, \alpha,\gamma,\varepsilon,\sigma_0})}{\hat \E(Z^{(2)}_{\tilde G, \alpha^*,\gamma^*,\varepsilon^*,\sigma_0})}  \\ & = \frac{\E(Z^{(2)}_{G_m, \alpha ,\gamma ,\varepsilon})}{\E(Z^{(2)}_{G_m, \alpha^* ,\gamma^* ,\varepsilon^* })} \times
\frac{\displaystyle\prod_{i=0}^{\gamma m d-1} \frac{(1-2\alpha+\gamma)dm+(d-1)M_1+(d-2)M_2-i}{(1-2\alpha+\gamma)dm-i}}
{\displaystyle\prod_{i=0}^{\gamma^* m d-1} \frac{(1-2\alpha^*+\gamma^*)dm+(d-1)M_1+(d-2)M_2-i}{(1-2\alpha^*+\gamma^*)dm-i}}\times\\
& \times \frac{\displaystyle\prod_{i=0}^{\gamma^* m d-1} \frac{dm+(d-1)M_1+(d-2)M_2-1-2i}{dm-1-2i}} {\displaystyle\prod_{i=0}^{\gamma m d-1} \frac{dm+(d-1)M_1+(d-2)M_2-1-2i}{dm-1-2i}}
  \times
\frac{\displaystyle\prod_{i=0}^{\varepsilon m d-1} \frac{(1-2\alpha)dm+(d-1)M_1+(d-2)M_2-i}{(1-2\alpha+\gamma)dm-i}}
{\displaystyle\prod_{i=0}^{\varepsilon^* m d-1} \frac{(1-2\alpha^*)dm+(d-1)M_1+(d-2)M_2-i}{(1-2\alpha^*)dm-i}}\times\\
& \times \frac{\displaystyle\prod_{i=0}^{(\alpha^*-\gamma^*) m d -1} \frac{(1-2\gamma^*)dm+(d-1)M_1+(d-2)M_2-1-2i}{(1-2\gamma^*)dm-1-2i}} {\displaystyle\prod_{i=0}^{(\alpha-\gamma) m d-1} \frac{(1-2\gamma)dm+(d-1)M_1+(d-2)M_2-1-2i}{(1-2\gamma)dm-1-2i}}
  \times \\
& \times  \frac{\displaystyle\prod_{i=0}^{\varepsilon m d -1} \frac{(1-2\alpha-\varepsilon)dm+(d-1)M_1+(d-2)M_2-i}{(1-2\alpha-\varepsilon)dm-i}}
{\displaystyle\prod_{i=0}^{\varepsilon^* m d -1} \frac{(1-2\alpha^*-\varepsilon^*)dm+(d-1)M_1+(d-2)M_2-i}{(1-2\alpha^*-\varepsilon^*)dm-i}}\times \\
& \times \frac {\displaystyle\prod_{i=0}^{\varepsilon^* m d-1} \frac{(1-2\alpha^*)dm+(d-1)M_1+(d-2)M_2-1-2i}{(1-2\alpha^*)dm-1-2i}}
{\displaystyle\prod_{i=0}^{\varepsilon m d-1} \frac{(1-2\alpha)dm+(d-1)M_1+(d-2)M_2-1-2i}{(1-2\alpha)dm-1-2i}}  \\
& = \exp\left( m\left(f(\alpha,\gamma,\varepsilon) - f(\alpha^*,\gamma^*,\varepsilon^*) \right) + O(\ln(m)) \right) \exp\left( O\left( n^{\frac 35 } \left(|\alpha-\alpha^*| +|\gamma-\gamma^*|+|\varepsilon - \varepsilon^*| \right)\right)\right)
\end{align*}
The final equality follows by Proposition \ref{prop:f-approximates-Z-square} and the bounds on the sizes of $m,M_1$ and $M_2$ by the assumed conditioning.
By Proposition~\ref{prop:f-global-max}, we have that for some constant $C$, possibly depending on $d$ and $\lambda$,
\[
f(\alpha ,\gamma ,\varepsilon ) - f(\alpha^*,\gamma^*,\varepsilon^*) \le -C \left( |\alpha - \alpha^*|^2 + |\gamma - \gamma^*|^2+ |\varepsilon - \varepsilon^*|^2  \right).
\]
Therefore, we obtain that
\begin{align*}
\frac{\hat \E(Z^{(2)}_{\tilde G, \alpha,\gamma,\varepsilon,\sigma_0})}{\hat \E(Z^{(2)}_{\tilde G, \alpha^*,\gamma^*,\varepsilon^*,\sigma_0})}
\le \exp\left(-C n\left( |\alpha-\alpha^*|^2+ |\gamma-\gamma^*|^2 +|\epsilon-\epsilon^*|^2\right)+ O \left(n^{\frac{3}{5}}\left( |\alpha-\alpha^*|+ |\gamma-\gamma^*| +|\varepsilon - \varepsilon^*|\right) \right) \right).
\end{align*}
\end{proof}

\begin{lemma}\label{lem:sigma-plus-sigma-B}
Let $(\alpha,\gamma,\varepsilon) \in \mathcal R$. For any $\sigma \in \{0,1\}^B$,
\begin{align*}
\frac{\hat \E(Z^{(2)}_{\tilde G, \alpha,\gamma,\varepsilon,\sigma})}{\hat \E(Z^{(2)}_{\tilde G, \alpha,\gamma,\varepsilon,\sigma_0})} = 
\left( \chi(\sigma)\right)^2\exp \left( O \left(|\alpha-\alpha^*|+ |\gamma-\gamma^*|+ |\varepsilon- \varepsilon^*|) n^{\frac{3}{5}}+ n^{\frac{1}{5}} \right) \right).
\end{align*}

\end{lemma}

\begin{proof}
 Using the expression \eqref{eq:secondmoment-tilde-G} for each of $\E(Z^{(2)}_{\tilde G, \alpha,\gamma,\varepsilon,\sigma_0})$ and $\E(Z^{(2)}_{\tilde G, \alpha,\gamma,\varepsilon,\sigma})$ and taking ratios of the numerators and denominators separately for each of the products, we have
\begin{align*}
& \frac{\hat \E(Z^{(2)}_{\tilde G, \alpha,\gamma,\varepsilon,\sigma})}{\hat \E(Z^{(2)}_{\tilde G, \alpha,\gamma,\varepsilon,\sigma_0})} = \lambda^{2|\sigma|}\exp \left( O(n^{\frac 15})\right) \times \\
  &   \times  \left( \frac{(1-2\alpha)^2}{(1-2\alpha+\gamma)(1-2\gamma)^{1/2}} \frac{(1-2\alpha - \varepsilon)}{(1-2\alpha)}
   \frac{(1-2\gamma)^{1/2}}{(1-2\alpha)^{1/2}}
   \frac{(1-2\alpha-2\varepsilon)}{(1-2\alpha-\varepsilon)}
   \frac{(1-2\alpha)^{1/2}}{(1-2\alpha-2\varepsilon)^{1/2}}
    \right)^{(d-1)L_1 + (d-2)L_2}  \\
& =   \lambda^{2|\sigma|}  \exp \left( O(n^{\frac 15})\right) \cdot \left(\frac{(1-2\alpha)(1-2\alpha-2\varepsilon)^{1/2}}{(1-2\alpha+\gamma)} \right)^{(d-1)L_1 + (d-2)L_2}\\
& = \lambda^{2|\sigma|}  \exp \left( O \left(n^{\frac 15}  +  n^{\frac{3}{5}} \left( |\alpha-\alpha^*|-|\gamma- \gamma^*| + |\varepsilon-\varepsilon^*| \right) \right)  \right) \left(\frac{(1-2\alpha^*)(1-2\alpha^*- 2\varepsilon^*)^{1/2}}{(1-2\alpha^*+ \gamma^*)} \right)^{(d-1)L_1 + (d-2)L_2}
\end{align*}
Since $\gamma^* = (\alpha^*)^2$, $\varepsilon ^*= \alpha^*(1-2\alpha^*)$ and  $ |\sigma| =L_1+L_2$, the last line gives that
\begin{align*}
& \frac{\hat \E(Z^{(2)}_{\tilde G, \alpha,\gamma,\varepsilon,\sigma})}{\hat \E(Z^{(2)}_{\tilde G, \alpha,\gamma,\varepsilon,\sigma_0})} = \left(\lambda\left(\frac{1-2\alpha^*}{1-\alpha^*}\right)^{d-1}\right)^{2|\sigma|} \exp \left( O \left( |\alpha-\alpha^*|+|\gamma-\gamma^*|+ |\varepsilon- \varepsilon^*|) n^{\frac{3}{5}}+ n^{\frac{1}{5}} \right) \right)
\end{align*}
and the lemma follows.
\end{proof}

Putting these results together we now prove Proposition \ref{prop:ratio-second-moments}.

\begin{proof}[Proof of Proposition \ref{prop:ratio-second-moments}]
Let $\sigma \in \{0,1\}^B$ and write
\begin{align*}
\frac{\hat \E\left(Z^{(2)}_{\tilde G, \alpha, \gamma,\varepsilon, \sigma}\right)}{\hat \E\left(Z^{(2)}_{\tilde G, \alpha^*, \gamma^*,\varepsilon^*, \sigma}\right)} =
\frac{\hat \E\left(Z^{(2)}_{\tilde G, \alpha, \gamma,\varepsilon, \sigma}\right)}{\hat \E\left(Z^{(2)}_{\tilde G, \alpha, \gamma,\varepsilon, \sigma_0}\right)}
\frac{\hat \E\left(Z^{(2)}_{\tilde G, \alpha, \gamma,\varepsilon, \sigma_0}\right)}{\hat \E\left(Z^{(2)}_{\tilde G, \alpha^*, \gamma^*,\varepsilon^*, \sigma_0}\right)}
\frac{\hat \E\left(Z^{(2)}_{\tilde G, \alpha^*, \gamma^*,\varepsilon^*, \sigma_0}\right)}{\hat \E\left(Z^{(2)}_{\tilde G, \alpha^*, \gamma^*,\varepsilon^*, \sigma}\right)}.
\end{align*}
Applying Lemmas \ref{lem:approximation-to-*} and \ref{lem:sigma-plus-sigma-B} to the terms in the product above, we have that 
\begin{align}
&\frac{\hat \E\left(Z^{(2)}_{\tilde G, \alpha, \gamma,\varepsilon, \sigma}\right) }{\hat \E\left(Z^{(2)}_{\tilde G, \alpha^*, \gamma^*,\varepsilon^*, \sigma}\right)}
\le  \nonumber \\ 
& \le \exp\left(O\left(n^{\frac 35}\left(|\alpha-\alpha^*|+|\gamma-\gamma^*|+|\varepsilon-\varepsilon^*|\right)+n^{\frac 15}\right) - Cn\left( |\alpha-\alpha^*|^2+|\gamma-\gamma^*|^2+|\varepsilon-\varepsilon^*|^2 \right)\right) . \label{eq:sigmaB-star-vs-nostar}
\end{align}

For a constant $C$, define the set
\[
\mathcal{R}_C = \left\{ (\alpha,\gamma,\varepsilon) \in \mathcal R \text{ \ s.t. \ } |\alpha-\alpha^*|, |\gamma-\gamma^*|,|\varepsilon-\varepsilon^*| \le C n^{-\frac{2}{5}} \right\}.
\]
Note that for some sufficiently large $C$, if $(\alpha,\gamma,\varepsilon) \in \mathcal R_{C}$, the the right-hand side of \eqref{eq:sigmaB-star-vs-nostar} can be bounded by $\exp\left(O(n^{\frac 15})\right)$. On the other hand, if 
$(\alpha,\gamma,\varepsilon) \not\in \mathcal R_{C}$ for any constant $C$, then the right-hand side of \eqref{eq:sigmaB-star-vs-nostar} can be made arbitraily small and therefore, for any $(\alpha,\gamma,\varepsilon) \in \mathcal R$ and $\sigma$
\[
\hat \E\left(Z^{(2)}_{\tilde G, \alpha, \gamma,\varepsilon, \sigma}\right) \le  \exp\left(O(n^{\frac 15})\right) \hat \E\left(Z^{(2)}_{\tilde G, \alpha^*, \gamma^*,\varepsilon^*, \sigma}\right).
\]
\end{proof}

\begin{proposition}\label{prop:tildeG-sigmaB-sigma0-secmom}
For any $\sigma \in \{0,1\}^B$,
\[
\hat \E\left(Z_{\tilde G,\sigma}^2 \right) \le \exp\left( O(n^{\frac 15})\right) \left(\chi(\sigma) \right)^2 \hat \E\left(Z_{\tilde G,\sigma_0}^2 \right).
\]
\end{proposition}
\begin{proof}Applying the Cauchy-Schwarz inequality, Proposition \ref{prop:ratio-second-moments} and Lemma \ref{lem:sigma-plus-sigma-B}, we have
\begin{align*}
\hat \E\left(Z_{\tilde G,\sigma}^2 \right) & = \hat \E \left( \left( \sum_\alpha Z_{\tilde G,\alpha,\sigma}\right)^2\right) = \tilde \Theta(1) \sum_\alpha \hat \E\left(Z_{\tilde G,\alpha,\sigma}^2 \right) = 
\tilde \Theta(1) \sum_{\alpha,\gamma,\varepsilon} \hat \E\left(Z_{\tilde G,\alpha,\gamma,\varepsilon,\sigma}^{(2)}  \right) \le \\ 
& \le \exp \left( O(n^{\frac 15})\right) \hat \E\left(Z_{\tilde G,\alpha^*,\gamma^*,\varepsilon^*,\sigma}^{(2)} \right) =  \exp \left( O(n^{\frac 15})\right) \left(\chi(\sigma) \right)^2 \hat \E\left(Z_{\tilde G,\alpha^*,\gamma^*,\varepsilon^*,\sigma_0}^{(2)} \right) \le \\
& \le \exp \left( O(n^{\frac 15})\right) \left(\chi(\sigma) \right)^2 \sum_{\alpha,\gamma,\varepsilon} \hat \E\left(Z_{\tilde G,\alpha,\gamma,\varepsilon,\sigma_0}^{(2)} \right) = \exp \left( O(n^{\frac 15})\right) \left(\chi(\sigma) \right)^2 \sum_{\alpha}\hat  \E\left(Z_{\tilde G,\alpha,\sigma_0}^{2} \right) \le \\
& \le \exp \left( O(n^{\frac 15})\right) \left(\chi(\sigma) \right)^2  \hat \E\left( \left( \sum_{\alpha} Z_{\tilde G,\alpha,\sigma_0}\right)^{2} \right) = 
\exp\left( O(n^{\frac 15})\right) \left(\chi(\sigma) \right)^2 \hat \E\left(Z_{\tilde G,\sigma_0}^2 \right).
\end{align*}
\end{proof}

The next step is to show a bound on the second moment of $Z_{\tilde G,\sigma_0}$ by the square of the first moment, and we begin with the following intermediate result.
\begin{lemma}
\label{lem:tildeG-secmom-vs-firstmom-at*}
\begin{align*}
\hat \E\left(Z^{(2)}_{\tilde G,\alpha^*,\gamma^*,\varepsilon^*,\sigma_0}\right) = \exp\left( O(n^{\frac 15})\right) \left(\hat \E(Z_{\tilde G,\alpha^*,\sigma_0})\right)^2.
\end{align*}

\end{lemma}

\begin{proof}
As before, we note that for the configuration $\sigma_0$, for $i=1,2$, $L_i-0$ and $K_i = M_i$. Comparing the expressions \eqref{eq:second-moment-product} and \eqref{eq:secondmoment-tilde-G} we obtain that

\begin{align}
\nonumber & \frac{\hat \E\left(Z^{(2)}_{\tilde G,\alpha^*,\gamma^*,\varepsilon^*,\sigma_0}\right)}{ \E\left( Z^{(2)}_{G_m,\alpha^*,\gamma^*,\varepsilon^*}\right)} = \exp\left( O(n^{\frac 15})\right) \times \\  
\nonumber &  \times \left(\frac{(1-2\alpha^*+\gamma^*)(1-2\gamma^*)^{\frac 12}}{(1-2\alpha^*)}
\frac{(1-2\alpha^*)(1-2\alpha^*)^{\frac 12}}{(1-2\alpha^*-\varepsilon^*)(1-2\gamma^*)^{\frac 12}}
\frac{(1-2\alpha^*)(1-2\alpha^*-\varepsilon^*)^{\frac12}}{(1-2\alpha^*-2\varepsilon^*)(1-2\alpha^*)^{\frac 12}}
 \right)^{(d-1)K_1 + (d-2)K_2 } \\
 & =  \exp\left( O(n^{\frac 15})\right) \left(\frac{(1-\alpha^*)^2}{1-2\alpha^*}\right)^{(d-1)K_1 + (d-2)K_2} \label{eq:tildeG-Gm-secmom}
\end{align}
where the last equality follows by canceling terms and using the fact that $\gamma^* = (\alpha^*)^2$ and $\varepsilon^*  = \alpha^*(1-2\alpha^*)$.
Similarly, comparing \eqref{eq:first-moment-product} and \eqref{eq:fraction}, we obtain that
\begin{align}
\frac{\hat \E(Z_{\tilde G,\alpha^*,\sigma_0})}{\E(Z_{G_m,\alpha^*})}  = \exp\left( O(n^{\frac 15})\right) \left(\frac{1-\alpha^*}{(1-2\alpha^*)^{\frac 12}}\right)^{(d-1)K_1 + (d-2)K_2} . \label{eq:tildeG-Gm-firstmom}
\end{align}

Combining Proposition \ref{prop:f-approximates-Z-square}, Lemma \ref{prop:Z-alpha-first-moment} and Lemma \ref{lem:second-moment-square-of-first}, we have that 
\begin{align}
\frac{\hat \E\left( Z^{(2)}_{G_m,\alpha^*,\gamma^*,\varepsilon^*} \right)}{\left( \hat \E(Z_{G_m,\alpha^*}) \right)^2} = \exp\left( O(n^{\frac 15})\right) \frac{\exp\left(m f(\alpha^*,\gamma^*,\varepsilon^*) + O(\ln m) \right) }{\exp\left(2n\Phi(\alpha^*) \right)} = \exp\left( O(n^{\frac 15})\right). \label{eq:Gm-secmom-firstmom}
\end{align}
Finally, putting together \eqref{eq:tildeG-Gm-secmom}, \eqref{eq:tildeG-Gm-firstmom} and \eqref{eq:Gm-secmom-firstmom} proves the lemma.

\end{proof}

\begin{proposition}\label{prop:tildeG--sigma0-secmom-bound}
\begin{align*}
\hat \E\left( Z_{\tilde G,\sigma_0}^2\right) \le \exp \left( O(n^{\frac 15})\right) \left(\hat \E\left( Z_{\tilde G,\sigma_0}\right) \right)^2
\end{align*}
\end{proposition}
\begin{proof}Applying the Cauchy-Schwartz inequality, we have
\begin{align*}
\hat \E\left( Z_{\tilde G,\sigma_0}^2\right) & = \hat \E\left( \left(\sum_\alpha Z_{\tilde G,\alpha, \sigma_0}\right)^2\right) \le \tilde \Theta(1) \sum_\alpha \hat \E\left(Z^2_{\tilde G,\alpha,\sigma_0} \right) = \tilde \Theta(1) \sum_{\alpha,\gamma,\varepsilon} \hat \E\left(Z^{(2)}_{\tilde G,\alpha,\gamma,\varepsilon,\sigma_0} \right) \le \\
& \le \exp \left( O(n^{\frac 15})\right) \hat \E\left( Z^{(2)}_{\tilde G,\alpha^*,\gamma^*,\varepsilon^*,\sigma_0} \right) \le 
  \exp \left( O(n^{\frac 15})\right)  \left(\hat \E\left(Z_{\tilde G,\alpha^*,\sigma_0} \right)\right)^2 \le \\
  & \le \exp \left( O(n^{\frac 15})\right)  \left(\hat \E\left(Z_{\tilde G,\sigma_0} \right)\right)^2 
\end{align*}
where the second inequality is by Proposition \ref{prop:ratio-second-moments} and the third inequality is by Lemma \ref{lem:tildeG-secmom-vs-firstmom-at*}.
\end{proof}
Define $Z_{G,\sigma}$
 to be the partition function over independent sets of $G$ whose restriction to $B$ is $\sigma$.
Extending this, define $Z_{G,\alpha,\sigma}$
 to be the partition function over independent sets of $G$ whose restriction to $B$ is $\sigma$ and for which $\alpha$ fraction of the vertices in $V(\tilde G)\setminus B$ are in the independent set. That is,

\[
Z_{G, \alpha, \sigma} \defeq \sum_{\omega \in I(G) : \omega(B) = \sigma, \sum_{v \in V(\tilde G) \setminus B}\omega_v = \alpha m} \lambda^{|\omega|}.
\]

\begin{lemma}\label{lem:markov-G-tilde}
For any $\sigma \in \{0,1\}^{B}$ and $0 \le \alpha < \frac 12$, the partition functions for $G$
  and $\tilde{G}$ can be related by
\[
Z_{G,\sigma} = \kappa(\sigma) Z_{\tilde G,\sigma}\] and
\[Z_{G,\alpha,\sigma} = \kappa(\sigma) Z_{\tilde G,\alpha,\sigma} \] 
\end{lemma}
where $\kappa(\sigma)$ is a constant depending only on the
configuration $\sigma$ and has a product structure
\[
\kappa (\sigma) = \displaystyle\prod_{i=1}^{k+1}\kappa_i(\sigma(W_i)).
\]
and $\kappa_i =  \kappa_j$ for $1 \le i, j \le k$.
\begin{proof}
By the Markov property.
\end{proof}

Putting these results together, we obtain the following.
\begin{proposition}\label{prop:sec-moment-G-sigma}For any $\sigma \in \{0,1\}^B$,
\begin{align*}
\hat \E((Z_{G,\sigma})^2)  \le  \exp(O(n^{\frac 15})) (\hat \E(Z_{G,\sigma}))^2.
\end{align*}

\end{proposition}
\begin{proof}

\begin{align*}
\nonumber \hat \E((Z_{G,\sigma})^2) & =  \hat \E((Z_{\tilde G,\sigma})^2
(\kappa(\sigma))^2)  &  \text{(by Lemma \ref{lem:markov-G-tilde})}\\
\nonumber & \le \exp\left(O(n^{\frac 15})\right)  \hat \E((Z_{\tilde G,{\sigma_0}})^2
(\chi(\sigma)\kappa(\sigma))^2) & \text{(by Proposition \ref{prop:tildeG-sigmaB-sigma0-secmom} )} \\
\nonumber & \le \exp\left(O(n^{\frac 15})\right) ( \hat \E(Z_{\tilde
  G,{\sigma_0}}\chi(\sigma)\kappa(\sigma) ))^2 \qquad \nonumber & \text{(by
 Proposition \ref{prop:tildeG--sigma0-secmom-bound})} \\
& \le   \exp\left(O(n^{\frac 15})\right)( \hat \E(Z_{G,\sigma}))^2. & \text{(by Corollary \ref{cor:tildeG-sigmaB-sigma0-firstmom} and Lemma \ref{lem:markov-G-tilde})} 
\end{align*}
\end{proof}

\section{Local Weak Convergence to the Free Measure on the Tree}\label{sec:local-weak-convergence}

The first result in this section shows that there does not exist a ``bad
set'' of neighborhoods with large stationary probability where the partition function is much larger than the
expected partition function.

\begin{proposition}\label{prop:E1-E2}
Let $c>0$ and suppose that $G \sim  \hat{\mathcal H}(n,d)$. The probability that there exists a set of independent set configurations $\mathcal B \subset
\{0,1\}^{B}$ such
that
\begin{align}
\mathcal E_1(\mathcal B):= \left \{ \displaystyle\sum_{\sigma \in \mathcal B}Z_{G,\sigma} >
\exp\left(cn^{\frac 35}\right)\displaystyle\sum_{\sigma \in \mathcal B} \hat \E\left(Z_{G,\sigma}\right) \right\}
\label{eq:bigger-than-avg} 
\end{align}
and
\begin{align}
\mathcal E_2(\mathcal B):= \left \{ \displaystyle\sum_{\sigma \in \mathcal B}Z_{G,\sigma} >
\exp\left(-n^{\frac 47}\right) \hat \E\left(Z_{G}\right) \right\}
\label{eq:bad-large-mass}
\end{align}
is at most $ \exp\left(-n^{\frac 35}(c-o(1))\right)$.
\end{proposition}
The choice of $\frac 47$ could be replaced with any constant less than $\frac 35$ and greater than $\frac 12$.
\begin{proof}
Suppose that there is a set $\mathcal B$ of boundary configurations such that $\mathcal E_1(\mathcal B)$ and $\mathcal E_2(\mathcal B)$ hold.
Define the set of configurations
\[
\mathcal D \defeq \left\{\sigma \in \mathcal B \ : \ Z_{G,\sigma} >
  \frac{1}{3} \exp\left(cn^{\frac 35}\right) \hat \E\left(Z_{G,\sigma}\right) \right\}.
\]
Suppose it was the case that
\begin{align*}
\displaystyle\sum_{\sigma \in \mathcal B \setminus \mathcal D}Z_{G,\sigma}
> \frac{1}{2}\displaystyle\sum_{\sigma \in \mathcal B}Z_{G,\sigma}.
\end{align*}

Then,
\begin{align*}
\frac{1}{3}\exp\left(cn^{\frac 35}\right)\displaystyle\sum_{\sigma \in \mathcal B
  \setminus \mathcal D} \hat \E(Z_{G,\sigma})
> \displaystyle\sum_{\sigma \in \mathcal B \setminus \mathcal
  D}Z_{G,\sigma} > \frac{1}{2}\displaystyle\sum_{\sigma \in \mathcal B}Z_{G,\sigma}.
\end{align*}
This contradicts \eqref{eq:bigger-than-avg} (that $\mathcal E_1(\mathcal B)$ holds) and thus we may assume
that
\begin{align*}
\displaystyle\sum_{\sigma \in \mathcal D}Z_{G,\sigma}
> \frac{1}{2}\displaystyle\sum_{\sigma \in \mathcal B}Z_{G,\sigma}.
\end{align*}
Therefore, by \eqref{eq:bad-large-mass} (that $\mathcal E_2(\mathcal B)$ holds) we have
\begin{align}\label{eq:D-large-part-of-B}
\displaystyle\sum_{\sigma \in \mathcal D}Z_{G,\sigma}
> \frac{1}{2}\displaystyle\sum_{\sigma \in \mathcal B}Z_{G,\sigma}
> \frac{1}{2}\exp\left(-n^{\frac 47}\right)\displaystyle\sum_{\sigma \in \{0,1\}^B} \hat \E\left(Z_{G,\sigma}\right).
\end{align}

By \eqref{eq:D-large-part-of-B}, and Markov's inequality,
\begin{align} \label{eq:prob-e1-e2}
 \hat \P\left(\exists \mathcal B \ \text{s.t.} \ \mathcal E_1(\mathcal  B) \cap \mathcal E_2(\mathcal  B) \right)&  \le  \hat \P\left( \displaystyle\sum_{\sigma \in \mathcal D}
Z_{G,\sigma} > \frac 12 \exp\left(-n^{\frac 47}\right)\sum_{\sigma \in \{0,1\}^B} \hat \E\left(Z_\sigma\right)\right) \le \frac{2 \hat \E\left(\displaystyle\sum_{\sigma \in \mathcal D}
Z_{G,\sigma}\right)}{ \exp\left(-n^{\frac 47}\right)\displaystyle\sum_{\sigma \in \{0,1\}^B} \hat \E\left(Z_\sigma\right)}.
\end{align}
By the definition of $\mathcal D$ and Proposition \ref{prop:sec-moment-G-sigma} we have
\begin{align}\label{eq:expectation-D}
\nonumber  \hat \E\left( \displaystyle\sum_{\sigma \in \mathcal D}Z_{G,\sigma} \right) & <
3 \exp\left(-cn^{\frac 35}\right)\displaystyle\sum_{\sigma \in \mathcal
  D} \hat \E\left( Z_{G,\sigma} \frac{Z_{G,\sigma}}{ \hat \E\left(Z_{G,\sigma}\right)} \right)  \le 3\exp\left(-cn^{\frac 35}\right)\displaystyle\sum_{\sigma \in \{0,1\}^B} \hat \E\left(
  Z_{G,\sigma} \frac{Z_{G,\sigma}}{ \hat \E\left(Z_{G,\sigma}\right)} \right) \le \\
  & \le \exp\left(-cn^{\frac 35}\right)\exp\left(O(n^{\frac 15})\right) \displaystyle\sum_{\sigma \in \{0,1\}^B}  \hat \E\left(Z_{G,\sigma}\right).
\end{align}
Putting together \eqref{eq:prob-e1-e2} and \eqref{eq:expectation-D}, we obtain that
\begin{align*}
 \hat \P\left(\exists \mathcal B \ \text{s.t.} \ \mathcal E_1(\mathcal  B) \cap \mathcal E_2(\mathcal  B)\right) \le \exp\left(-n^{\frac 35}(c-o(1))\right),
\end{align*}
which completes the proof of Proposition \ref{prop:E1-E2}.
\end{proof}

Recall the definition of the set of vertices $S'$ from \eqref{eq:centers-S} and recall that for each $s_i \in S'$, $W_i$ is the set of vertices on the boundary $\partial B_r(s_i)$ and $|S'|=k$. Fix $1 \le i \le |S'|$. Let $T_1,\ldots,T_{|W_i|}$ be $(d-1)$-ary trees and let $\tilde T$ be their union. Let $\P_{\tilde T}$ be the product of the free measures on the trees. Let us identify the roots $u^{(i)} = \{u_1,\ldots,u_{|W_i|}\}$ of these trees  with the vertices of $W_i$. Let $T$ be the tree obtained by joining to $\tilde T$ a $d$-ary tree of depth $r$ whose leaves are identified with the $u_i$.

\begin{lemma}\label{lem:tree-measure-equals-graph-measure}
Define the distribution $\nu$ on $\{0,1\}^B$ by
\[
\nu(A): = \frac{\displaystyle \sum_{\sigma \in A}\chi(\sigma) \kappa(\sigma)}{\displaystyle \sum_{\sigma \in \{0,1\}^B} \chi(\sigma) \kappa(\sigma)}
\]
 If $\sigma \sim \nu$, then for each $1 \le i \le k$,
 the $\sigma(W_i)$ are independent and
\begin{align}\label{eq:graph-tree-measure}
\nu(\sigma(W_i) \in \cdot) = \P_{T}(\sigma(u^{(i)}) \in \cdot) = \P_{T_d}(B_r(\rho) \in \cdot).
\end{align}
\end{lemma}
\begin{proof}  By relating the occupation probability of the root for the free measure for the $d$-regular tree and the occupation probability of the root of the free measure on the  $(d-1)$-ary tree, it can be verified that
\[
\P_{\tilde T}(\sigma(W_i) = \omega) \propto \chi(\omega).
\]
By the Markov property,
\[
\P_{ T}(\sigma(W_i)= \omega) \propto \chi(\omega)\kappa(\omega).
\]
Therefore
\begin{align}\label{eq:nu-PT}
\nu(\sigma(W_i) = \omega) = \P_{ T}(\sigma(u^{(i)}) = \omega)
\end{align}
and \eqref{eq:graph-tree-measure} follows.
\end{proof}

Let $\P_{G_n}$ denote the hardcore measure on the random $d$-regular graph of size $n$. Recall that we have local weak convergence to the free measure if for all all $r$, with high probability over a uniformly chosen vertex $u$, for $\varepsilon >0$, as $n \to \infty$,
\begin{align}\label{eq:graph-tree-reconst-equiv}
 \P\left(d_{tv}(\P_{G_n}(\sigma({B_r}(u)) \in \cdot),
  \P_{T_d}(\sigma(B_r(\rho))\in \cdot) > \varepsilon ) \right) \to 0.
 \end{align}
 The following lemma will be used in the next result.
\begin{lemma}\label{lem:markov}
Let $I$ be an index set, $(X_i)_{i \in I}$ random variables and $(a_i)_{i \in I}$ constants. Suppose that for each $i$,
\[
\P(X_i < a_i) \le \varepsilon
\]
for some $\varepsilon >0$. Then,
\[
\P\left(\sum_{i \in I} X_i < \frac 12\sum_{i \in I} a_i\right) \le 2\varepsilon.
\]
\end{lemma} 
\begin{proof}We show first that 
\begin{align*}
\sum_{i} X_i < \frac 12\sum_{i} a_i \ \ \ \Rightarrow \ \ \ \sum_{i} a_i \mathds{1}_{X_i > a_i} < \frac 12 \sum_i a_i   \ \ \ \Leftrightarrow \ \ \ \sum_{i} a_i \mathds{1}_{X_i \le a_i} \ge \frac 12 \sum_i a_i.
\end{align*}
The equivalence above is immediate and the first implication can be seen as follows:
\begin{align*}
\frac 12\sum_{i} a_i  > \sum_{i} X_i \ge \sum_{i} X_i \mathds{1}_{X_i > a_i} \ge \sum_{i} a_i \mathds{1}_{X_i > a_i}. 
\end{align*}
Applying Markov's inequality, we have
\[
\P\left( \sum_{i} X_i < \frac 12\sum_{i} a_i \right) \le \P\left(\sum_{i} a_i \mathds{1}_{X_i \le a_i} \ge \frac 12 \sum_i a_i \right) \le 2\varepsilon.
\]
\end{proof}

We will show the following result which in turn implies \eqref{eq:graph-tree-reconst-equiv}, since with high probability  $k= (1-o(1))n^{3/5}$.

\begin{theorem}\label{thm:graph-tree-reconst-equiv}
Let $G \sim \mathcal H(n,d)$. Let $\omega$ be an independent set drawn according to the hardcore measure on $G$. For any $\varepsilon >0$,
\begin{align}\label{eq:**}
\frac{|\{ 1 \le i \le k \ : \  d_{tv}(\P_{G_n}(\omega(W_i)) \in \cdot), \P_{T_d}(B_r(\rho) \in \cdot)) > \varepsilon  \}|}{k} \stackrel{\P}{\to} 0
\end{align}
as $n \to \infty$.
\end{theorem}

\begin{proof} Let $\mathcal E$ be the event that the left hand side of \eqref{eq:**} is at least $\delta > 0$.  If $\mathcal E$ occurs, then by the definition of total variation distance, there exists a set $J$ of indices of size $\delta k$ and $A_i \in \{0,1\}^{W_i}$ for $i \in J$ such that
\begin{align*}
\P_G\left(\omega(W_i) \in A_i\right)- \P_{T}(\omega(u^{(i)}) \in A_i) >\varepsilon \ \ \ \ \forall i \in J.
\end{align*}

This implies
\begin{align}\label{eq:expected-num-bad-neighborhoods}
& \E\left(\displaystyle\sum_{i \in J} \mathds 1(\omega(W_i) \in A_i)\right)-  \displaystyle\sum_{i \in J} \P_{T}(\omega(u^{(i)}) \in A_i) \ge \varepsilon \delta
k .
\end{align}
Using the fact that $\displaystyle\sum_{i \in J} \mathbbm 1(\omega(W_i) \in A_i) \le \delta k$, we obtain that
\begin{align}
\nonumber  \E\left(\displaystyle\sum_{i \in J} \mathbbm 1(\omega(W_i) \in A_i)\right) & - \displaystyle\sum_{i \in J} \P_{T}(\omega(u^{(i)}) \in A_i) \le \\
  \le \frac{\varepsilon \delta n^{\frac 35}}{2} & + \delta k \P\left( \displaystyle\sum_{i \in J} \mathds 1 (\omega(W_i) \in A_i) >
\frac{\varepsilon \delta n^{\frac 35}}{2} + \displaystyle\sum_{i \in J}
\P_{T}(\omega(u^{(i)}) \in A_i)\right) . \label{eq:prob-bad-neighborhoods}
\end{align}
Combining \eqref{eq:expected-num-bad-neighborhoods} and \eqref{eq:prob-bad-neighborhoods} and using the fact that $k \ge (1-o(1))n^{\frac 35}$, we get
\begin{align}
 \label{eq:probability-of-mathcalB}
 \P\left( \displaystyle\sum_{i \in J} \mathds 1 (\omega(W_i) \in A_i) >
\frac{\varepsilon \delta n^{\frac 35}}{2} + \displaystyle\sum_{i \in J}
\P_{T}(\omega(u^{(i)}) \in A_i)\right) \ge \frac{\varepsilon }{3}.
\end{align}
Define the set of configurations
\[
 \mathcal B := \left \{ \sigma \in \{0,1\}^{\cup_i W_i} \mathrm{ \ \  s.t.  \ \ }  \displaystyle\sum_{i \in J} \mathbbm 1 (\sigma \in A_i) >
\frac{\varepsilon \delta n^{\frac 35}}{2} + \displaystyle\sum_{i \in J}
\P_{T}(\sigma \in A_i) \right \}.
\]
By \eqref{eq:probability-of-mathcalB}, on the event $\mathcal E$
\begin{align}\label{eq:ZGsigma-ZG}
\sum_{\sigma \in \mathcal B} Z_{G,\sigma} \ge \frac{\varepsilon}{3} Z_G.
\end{align}
In particular, \eqref{eq:ZGsigma-ZG} holds when $G \sim \hat{\mathcal H}(n,d)$. By Proposition \ref{prop:sec-moment-G-sigma}, for any $\sigma \in \{0,1\}^B$ and  $G \sim \hat{\mathcal H}(n,d)$,
\begin{align*}
\frac{\hat \E(Z_{G,\sigma}^2)}{(\hat\E(Z_{G,\sigma}))^2} \le \exp\left( O(n^{\frac 15})\right),
\end{align*}
so that by the Paley-Zygmund inequality,
\begin{align*}
\hat \P\left(Z_{G,\sigma} > \frac{1}{2} \hat\E(Z_{G,\sigma})\right) \ge \exp\left( -O(n^{\frac 15})\right).
\end{align*}
Indeed, by Markov's inequality, it follows that for any $\kappa >0$
\begin{align}
\label{eq:ZG-in-EZG-interval}
\hat \P\left(Z_{G,\sigma} \in \left[\frac{\hat\E(Z_{G,\sigma})}{2},e^{n^\kappa} \hat\E(Z_{G,\sigma})\right]\right) \ge \exp\left( -O(n^{\frac 15})\right).
\end{align}

On the other hand, by Azuma's inequality, for all $\kappa >0$,
\begin{align}\label{eq:Azuma}
\hat \P\left(|\ln Z_{G,\sigma} - \hat \E(\ln Z_{G,\sigma})| \ge n^{\frac 12 +\kappa}\right) \le \exp\left(-\frac{n^{1+2 \kappa}}{2 n d \lambda}\right) \le \exp(-n^{2 \kappa}).
\end{align}

Together, \eqref{eq:ZG-in-EZG-interval} and \eqref{eq:Azuma} imply that the interval
\[[ \ln(\hat \E(Z_{G,\sigma}) -\ln2,{n^\kappa } + \ln \hat \E(Z_{G,\sigma})] \cap [\hat \E(\log Z_{G,\sigma}) - n^{1/2+\kappa},\hat \E(\log Z_{G,\sigma}) + n^{1/2+\kappa}]\] is non-empty
and hence
\[
|\hat \E(\ln Z_{G,\sigma}) - \ln(\hat \E(Z_{G,\sigma}))| \le 2n^{\frac 12 + \kappa}.
\]
Plugging this into \eqref{eq:Azuma}, we obtain that for each fixed $\sigma$,
\[
\hat \P\left( Z_{G,\sigma} \ge \exp\left(-O(n^{\frac 12 + \kappa})\right)\hat \E(Z_{G,\sigma}) \right) \ge 1- \exp(-n^{2 \kappa}).
\]
By Lemma \ref{lem:markov},
\[
\hat \P\left( \sum_{\sigma \in \mathcal B} Z_{G,\sigma} \ge \exp\left(-O(n^{\frac 12 + \kappa})\right) \sum_{\sigma \in \mathcal B} \hat \E(Z_{G,\sigma}) \right) \ge 1- 2 \exp(-n^{2 \kappa})
\]
and combining with \eqref{eq:ZGsigma-ZG}, we have that
\[
\hat \P\left( \sum_{\sigma \in \mathcal B} Z_{G,\sigma} \ge \exp\left(-O(n^{\frac 12 + \kappa})\right)\hat \E(Z_{G}) \right) \ge 1- 2 \exp(-n^{2 \kappa}).
\]
That is, we have shown that the event $\mathcal E_2(\mathcal B)$, as defined in Proposition \ref{prop:E1-E2} holds with high probability on the event $\mathcal E$.
By Corollary \ref{cor:tildeG-sigmaB-sigma0-firstmom} and Lemma \ref{lem:markov-G-tilde} for each $\sigma \in \mathcal B$, we have
\[
\hat \E(Z_{G,\sigma}) = \exp\left( O(n^{\frac 15}) \right) \chi(\sigma) \kappa(\sigma) \hat \E(Z_{G,\sigma_0}).
\]
Summing this over all possible $\sigma $, we have
\[
\hat \E(Z_{G}) = \sum_{\sigma \in \{0,1\}^B }\exp\left( O(n^{\frac 15}) \right) \chi(\sigma) \kappa(\sigma) \hat \E(Z_{G,\sigma_0}),
\]
and comparing these two equalities, we obtain that
\begin{align}
\label{eq:ZG-bad-close-to-ZG-nuB}
\sum_{\sigma \in \mathcal B}\hat \E(Z_{G,\sigma}) = \exp\left( O(n^{\frac 15}) \right) \frac{ \displaystyle \sum_{\sigma \in \mathcal B}\chi(\sigma) \kappa(\sigma)}{\displaystyle \sum_{\sigma \in \{0,1\}^B} \chi(\sigma) \kappa(\sigma)} \hat \E(Z_{G}) = \exp\left( O(n^{\frac 15}) \right) \nu(\mathcal B) \hat \E(Z_{G}).
\end{align}
By 
Lemma \ref{lem:tree-measure-equals-graph-measure} and Azuma-Hoeffding,

\begin{align*}
\nu(\mathcal B) & = \nu \left( \displaystyle\sum_{i \in J} \mathbbm 1 (\sigma \in A_i) >
\frac{\varepsilon \delta n^{\frac 35}}{2} + \displaystyle\sum_{i \in J}
\P_{T}(\sigma \in A_i) \right) \\
&=\nu \left( \displaystyle\sum_{i \in J} \mathds 1 (\sigma \in A_i) - \nu \left (\displaystyle\sum_{i \in J} \mathds 1 (\sigma \in A_i) \right) >
\frac{\varepsilon \delta n^{\frac 35}}{2}  \right) \le \exp \left( -cn^{\frac 35}\right).
\end{align*}
Combining the above bound with \eqref{eq:ZG-bad-close-to-ZG-nuB}, we obtain that
\begin{align}\label{eq:EZG-bad-ub}
\sum_{\sigma \in \mathcal B}\hat \E(Z_{G,\sigma}) \le \exp\left( - cn^{\frac 35}\right) \hat \E(Z_G).
\end{align}

Using \eqref{eq:EZG-bad-ub}, we have that if $\mathcal E$ and $\mathcal E_2(\mathcal B)$ hold, then 
\[
\sum_{\sigma \in \mathcal B} Z_{G,\sigma}  >  \exp\left(-O(n^{\frac 12 + \kappa})\right)\hat\E(Z_G) \ge   \exp\left(cn^{\frac 35}\right) \hat \E \left( \sum_{\sigma \in \mathcal B} Z_{G,\sigma}  \right),
\]
and therefore, the event $\mathcal E_1(\mathcal B)$ as defined in Proposition \ref{prop:E1-E2} holds. Therefore, on the event $\{\mathcal E \cap \mathcal E_2(\mathcal B)\}$, $\mathcal E_1(\mathcal B)$ holds with probability at least $1-2\exp(-n^{2\kappa})$.

To summarise, we have shown that on the event $\mathcal E$, the event $\{\mathcal \E_1(\mathcal B) \cap \mathcal E_2(\mathcal B)\}$  holds with high probability.
Applying Proposition \ref{prop:E1-E2},
\begin{align*}
\exp(-c n^{\frac 35}) \ge \hat \P(\mathcal E_1(\mathcal B) \cap \mathcal E_2(\mathcal B)) \ge \hat \P(\mathcal E_1(\mathcal B) \cap \mathcal E_2(\mathcal B) | \mathcal E) \hat \P(\mathcal E) \ge (1-o(1)) \hat \P(\mathcal E).
\end{align*}
Thus, $\hat \P(\mathcal E) \to 0$ as $n \to \infty$. Since $G \sim \hat{\mathcal H}(n,d)$ with high probability, $ \P(\mathcal E) \to 0$ as $n \to \infty$ and the claim follows.
\end{proof}

Theorem \ref{thm:graph-tree-reconst-equiv} implies \eqref{eq:graph-tree-reconst-equiv}, which establishes local weak convergence to the free measure, proving Theorem~\ref{thm:local-weak-conv}.  Given local weak convergence, the equivalence of the reconstruction thresholds and hence Theorem~\ref{thm:graph-tree-reconst}, follow.

\section{Technical Lemmas about the partition function}\label{sec:tech-lemmas}
In this section we show that the second moment $\E(Z_{G}^2)$ is close
to the the square of the first moment $\E(Z_{G})$ and satisfies a quadratic decay property. Let
\[
H_1(x,y) = -x(\ln(x)-\ln(y))+(x-y)(\ln(y-x)-\ln(y)).
\]

Recall that
\begin{align}\label{eq:Z2-product}
\E\left(Z^{(2)}_{G,\alpha,\gamma, \varepsilon}\right) = \lambda^{2 \alpha n}{n\choose
  \alpha n} {\alpha n \choose \gamma n}
{(1-\alpha)n \choose (\alpha-\gamma)n}
\left(
\frac{\displaystyle\prod_{i=0}^{\gamma n d -1}
  (1-2\alpha+\gamma)dn - i}{\displaystyle\prod_{i=0}^{\gamma n d -1 } dn-1 - 2i}\right) \times
\nonumber \\
\times \left(
\frac{\displaystyle\prod_{i=0}^{\varepsilon d n -1}  (1-2\alpha)dn - i
  \cdot \displaystyle\prod_{i=0}^{(\alpha-\gamma-\varepsilon)dn -1}
  (\alpha-\gamma)dn - i}{\displaystyle\prod_{i=0}^{(\alpha-\gamma)d n -1}
  (1-2\gamma)dn-1 -2i}
\cdot  \frac{\displaystyle\prod_{i=0}^{\varepsilon d n -1}
  (1-2\alpha-\varepsilon)dn - i
}{\displaystyle\prod_{i=0}^{\varepsilon d n -1} (1-2\alpha)dn -1- 2i}
\right).
\end{align}

We also recall that the following function arises naturally in the estimation of the
second moment:
\[
f(\alpha,\gamma,\varepsilon) = 2\alpha
\ln(\lambda)+H(\alpha)+H_1(\gamma,\alpha)+ H_1(\alpha-\gamma,1-\alpha)
+d\Psi_2(\alpha,\gamma,\varepsilon)
\]
where
\begin{align}
& \Psi_2(\alpha,\gamma,\varepsilon) = H_1(\varepsilon,\alpha-\gamma)+
\int_0^\gamma
\ln(1-2\alpha+\gamma-x)d{x} - \int_0^\gamma \ln(1-2x)d{x} \nonumber \\
& + \int_0^{\varepsilon}
\ln(1-2\alpha-x)d{x} +
\int_0^{\alpha-\gamma-\varepsilon}\ln(\alpha-\gamma-x)d{x} -
\int_0^{\alpha-\gamma} \ln(1-2\gamma-2x)d{x} \nonumber \\
& +  \int_0^\varepsilon
\ln(1-2\alpha-\varepsilon-x)d{x} - \int_0^\varepsilon
\ln(1-2\alpha-2x)d{x}. \label{eq:Psi2}
\end{align}

Using \eqref{lem:product-f-comparison} to compare terms in \eqref{eq:Z2-product} and \eqref{eq:Psi2} proves Proposition \ref{prop:f-approximates-Z-square} showing that
\begin{align*}\label{e:fApproximation}
\E\left(Z^{(2)}_{\alpha,\gamma,\varepsilon}\right) = \exp(nf(\alpha,\gamma,\varepsilon) + O(\ln n)).
\end{align*}

Thus the second moment depends on the behavior of the function
$f$. We will show a series of technical lemmas showing that $f$ attains its maximum at $(\alpha^*,\gamma^*,\varepsilon^*)$ and decays quadratically around this point. 
Define
\[
\alpha_c := \frac{(2-\delta_d)\log(d)}{d}
 \]
where $\delta_d = C \frac{\ln \ln d +1}{\ln d} \to 0$ as $d \to \infty$ and $C>3$.
\begin{lemma}\label{lem:max-eps}
For each fixed $\alpha,\gamma$ in the region $\mathcal R$, the function $f$ has a local maximum at 
\[
\overline \varepsilon = \overline \varepsilon(\alpha,\gamma) =  \frac{1}{2} \left(   1-2\gamma -
\sqrt{(1-2\alpha)^2+4(\alpha-\gamma)^2} \right)
\]
\end{lemma}
\begin{proof}
Differentiating \eqref{eq:Psi2}, we have that the derivative of $f$ is given by
\[
\frac{\partial f}{\partial \varepsilon} = d\ln \left(
\frac{(\alpha-\gamma-\varepsilon)(1-2\alpha-\varepsilon)(1-2\alpha-2
  \varepsilon)^2}{\varepsilon^2 (1-2\alpha-\varepsilon)(1-2\alpha-2
  \varepsilon)} \right)
\]
Since the hardcore model is a permissive model, we may assume that the local maxima of $f$ are in the interior of $\mathcal R$ (see e.g. \cite[Proposition 3.2]{DemMonSun13}.
Solving for $\varepsilon$ by setting $\frac{\partial f}{\partial
  \varepsilon}= 0$, gives that the unique 
solution in the interior of $\mathcal R$ is $\varepsilon = \overline \varepsilon$. 
Further, we check that the second derivative
\[\frac{\partial^2 f}{\partial \varepsilon^2}=
-d\left( \frac{1}{\alpha-\gamma-\varepsilon}+\frac{2}{1-2\alpha
    -2\varepsilon}+\frac{2}{\varepsilon}\right)  <0
\]
and hence $\overline \varepsilon$ is a maximum. 
\end{proof}
We will also use the following technical lemma.
\begin{lemma}\label{lem:second-moment-square-of-first}
For any $0 \le \alpha \le \frac 12$
\begin{align*}
2\Phi(\alpha) = f(\alpha,\hat \gamma,\hat \varepsilon).
\end{align*}
\end{lemma}
\begin{proof}Substituting $\hat \gamma = \alpha^2$ and $\hat  \varepsilon = \alpha(1-2\alpha)$ in the expression for $f(\alpha,\gamma,\varepsilon)$, and simplifying, we have
\begin{align*}
2 \Phi(\alpha)  - f(\alpha,\hat  \gamma,\hat  \varepsilon)
 = \ & 2d \left( (1-\alpha)\ln(1-\alpha) - \frac{1-2\alpha}{2} \ln \left( 1-2\alpha\right) -2 \right) -d\Psi_2(\alpha,\hat  \gamma,\hat  \varepsilon) \\
& + H(\alpha) + H_1(\alpha^2,\alpha)+H_1(\alpha(1-\alpha),1-\alpha) \; = \; 0 \qedhere
\end{align*}
\end{proof}

Define the function
\[
g(\alpha,\gamma) \equiv f(\alpha,\gamma,\overline \varepsilon(\alpha,\gamma))
\]
and consider its extremal values for fixed $\alpha$ in the region
where $\alpha,\gamma \ge 0$ and $\alpha -\gamma \ge 0$.
The derivative of $g$ is given by
\begin{align}\label{e:gFirstDerivative}
\frac{\partial g}{\partial \gamma}(\alpha,\gamma) = \frac{\partial
  f}{\partial \gamma} (\alpha,\gamma,\overline \varepsilon(\alpha,\gamma))
\end{align}
while its second derivative is
\begin{align}\label{e:gSecondDerivative}
\frac{\partial^2 g}{\partial \gamma^2}(\cdot,\cdot) = \frac{\partial
  f}{\partial^2 \gamma}(\cdot,\overline \varepsilon) + \frac{\partial \overline
  \varepsilon}{\partial \gamma} \frac{\partial f}{ \partial \gamma
  \partial \varepsilon}(\cdot,\cdot,\overline \varepsilon).
\end{align}

We  establish  the behavior of the function $f$ near its maximum by showing several facts about $g$.

\begin{proof}[Proof of Proposition \ref{prop:f-global-max}]
By Lemma \ref{lem:max-eps}, for fixed $\alpha$ and $\gamma$, $f$ has a maximum at $\overline \varepsilon(\alpha,\gamma)$. Thus, it remains to find the maximum of $g(\alpha,\gamma)$. In what follows we will show that for fixed $\alpha$, $g$ is maximized at $(\alpha,\alpha^2)$. 
Noting that  $\overline \varepsilon (\alpha,\hat \gamma) = \alpha(1-2\alpha)$, it follows that $f$ is maximized at $(\alpha^*, \gamma^*,\varepsilon^*)$ since by Lemma \ref{lem:second-moment-square-of-first}, $f(\alpha,\alpha^2,\alpha(1-2\alpha)) = 2\Phi(\alpha)$, and by \eqref{eq:alpha-star-eqn}, $\Phi$ is maximized at $\alpha^*$.

We can verify that $(\alpha,\alpha^2)$ is a stationary
point of $g$ by computing the first  derivative.
In what follows, we establish that for fixed $\alpha < \alpha_c$, $(\alpha,\alpha^2)$ is a global maximizer of $g(\alpha,\gamma)$ by considering several possible ranges for $\gamma$. This will consist of two main steps. The first is to show that $(\alpha,\alpha^2)$ is a maximum and the second is to show that the function $g$ is larger at $(\alpha,\alpha^2)$ than at any other possible maximum. Let $\gamma_i:=c_i \frac{\alpha}{\ln(\alpha^{-1})}$ for $i =1,2,3$ with the constants $c_i$ to be set later. 

Computing the derivatives
using equations~\eqref{e:gFirstDerivative} and~\eqref{e:gSecondDerivative} gives that

\begin{align}
\frac{\partial^2 g}{\partial \gamma^2} = & -\left(
\frac{2}{\alpha-\gamma}+\frac{1}{\gamma} + \frac{1}{1-2\alpha+\gamma}
\right) \nonumber \\
& + d\left( \frac{1}{1-2\alpha+\gamma}+ \frac{2}{\alpha-\gamma}-
\frac{2(\alpha-\gamma)}{(\alpha-\gamma-\overline\varepsilon)
  \sqrt{(1-2\alpha)^2+4(\alpha-\gamma)^2}} \right). \label{eq:d2g-by-dgamma2}
\end{align}

\begin{figure}[h]
\center
\includegraphics[width=9cm]{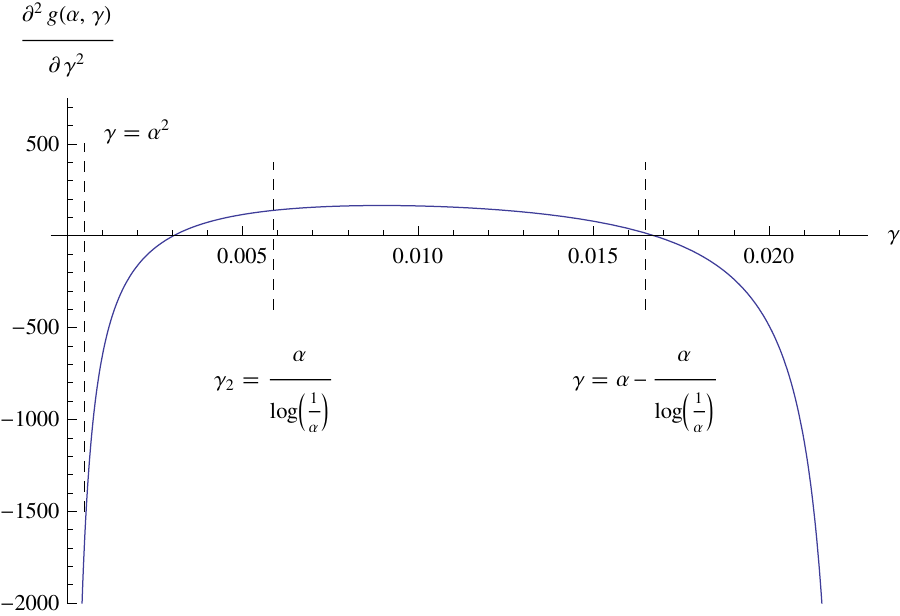}
\caption{Second derivative of $g(\alpha,\gamma)$ with respect to $\gamma$.}
\label{fig:del2-g}
\end{figure}

\begin{figure}[h]
\center
\includegraphics[width=9cm]{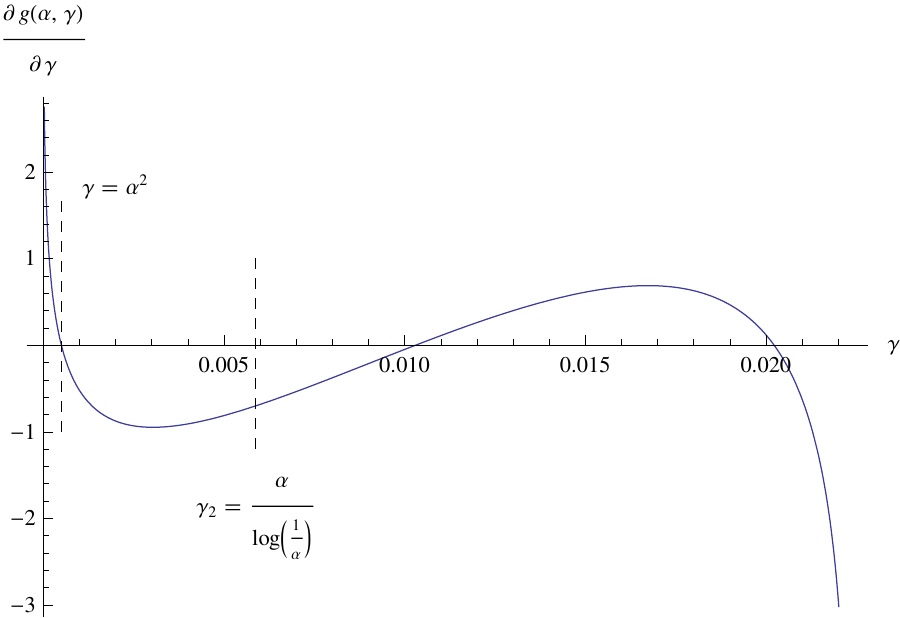}
\caption{First derivative of $g(\alpha,\gamma)$ with respect to $\gamma$.}
\label{fig:del-g}
\end{figure}

\begin{figure}
\center
\includegraphics[width=9cm]{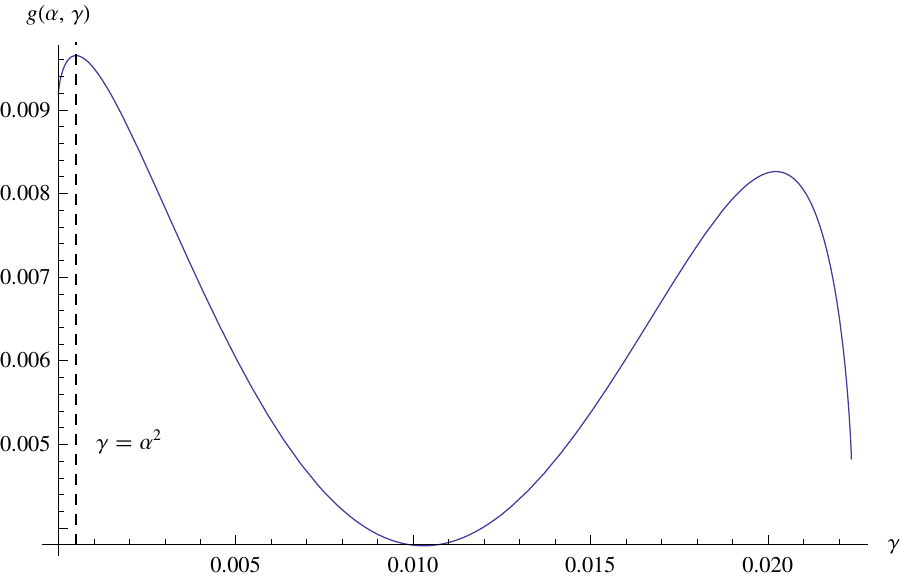}
\caption{The function $g(\alpha,\gamma)$.}
\label{fig:g}
\end{figure}

\begin{enumerate}[1)]
\item \label{it:alpha-square-is-max}
Let $d$ be sufficiently large so that by the assumption that $\alpha < \alpha_c$, we have that $d < 3 \alpha^{-1}\ln(\alpha^{-1})$. Let $d$ also be large enough so that for a small constant $c_1$ to be chosen later, $\alpha^2 < c_1 \alpha/ \ln(\alpha^{-1})$. 
We show that there is a constant $c_1$ so that for $\gamma \in \left[0,\gamma_1\right]$, $\frac{\partial^2 g}{\partial \gamma^2} <0 $ (see Figure \ref{fig:del2-g}) and hence the stationary point
  $(\alpha,\alpha^2)$ of $g$ is a maximum (see Figures Figure \ref{fig:del-g} and \ref{fig:g}).

Note that the first term of the second derivative \eqref{eq:d2g-by-dgamma2} is
  negative and has
  magnitude at least $1/\gamma \ge \alpha^{-1}\ln(\alpha^{-1})/c_1$ for this range of $\gamma$. The term on the second line of \eqref{eq:d2g-by-dgamma2} is positive, but we will argue that its magnitude is
  $O(\alpha^{-1}\ln(\alpha^{-1}))$. The claim will follow by taking the constant $c_1$ to be small enough.

The terms in the second bracket of \eqref{eq:d2g-by-dgamma2} can be
bounded by $O(1)$ as can be seen below from the series expansion in Mathematica (note that this calculation does not depend on the size of $\gamma$ or on $c_1$).
\begin{align}
\nonumber & \frac{1}{1-2\alpha+\gamma}+ \frac{2}{\alpha-\gamma} -  \frac{2 (\alpha-\gamma)}{\sqrt{(1-2 \alpha)^2+4 (\alpha-\gamma)^2} \left(\alpha-\gamma-\frac{ \left(1-2\gamma-\sqrt{(1-2 \alpha)^2+4 (\alpha-\gamma)^2}\right)}{2}\right)} \\
= & \left(1+4 \alpha+O(\alpha^2)\right)+\left(-3-12 \alpha+O(\alpha^2)\right) \gamma+O(\gamma^2) \label{eq:second-term-expansion}
\end{align}
Therefore
\[
\frac{\partial^2 g}{\partial \gamma^2} < -\frac 1\gamma + C d < - \frac{1}{c_1} \alpha^{-1} \ln(\alpha^{-1})  + C \alpha^{-1} \ln(\alpha^{-1})
\]
where above the particular value of $C$ may change in each appearance.
 Hence by choosing $c_1$ to be sufficiently small, the claim follows.

We now divide the analysis showing that $\gamma = \alpha^2$ corresponds to a global maximum into two cases based on the size of $\alpha$.

\item When $\alpha$ is small enough, we show that $g$ has no stationary point for $\gamma \in (\alpha^2,\alpha]$. 
Suppose that $\alpha < \varepsilon \ln(d) / d$ for $\varepsilon$ sufficiently small. Then, for some $\varepsilon'$, $d< \varepsilon' \alpha^{-1} \ln(\alpha^{-1})$. Expanding the terms in the second bracket of \eqref{eq:d2g-by-dgamma2} as in \eqref{eq:second-term-expansion}, and  recalling that $\partial g(\alpha,\alpha^2)/\partial \gamma~=~0$, we have that for $\gamma > \alpha^2$,
\begin{align}\label{eq:small-alpha-dg-upper-bound}
\frac{\partial g}{\partial \gamma} (\alpha,\gamma)   & <  \int_{\alpha^2}^\gamma 
(-\frac{1}{s} + Cd) d s < \ln\left(\frac{\alpha^2}{\gamma}\right) + C\varepsilon' \alpha^{-1}\ln(\alpha^{-1}) (\gamma-\alpha^2) .
\end{align}

We claim that the bound on the right hand side of \eqref{eq:small-alpha-dg-upper-bound} is maximized at the end points of the interval $[\alpha^2,\alpha]$.
Indeed, differentiating the bound with respect to $\gamma$, the only stationary point in the interval is at $\gamma = \alpha/(C \varepsilon' \ln(\alpha^{-1}) )$. Furthermore, the second derivative of the bound is positive so that the stationary point can only be a minimum.

At $\gamma = \alpha^2$, the bound evaluates to $0$. Evaluating \eqref{eq:small-alpha-dg-upper-bound} at $\gamma= \alpha$, we obtain that
\[
\frac{\partial g}{\partial \gamma} (\alpha,\alpha)  < \ln(\alpha) + C\varepsilon' \alpha^{-1}\ln(\alpha^{-1}) (\alpha-\alpha^2) \le \ln(\alpha) (1-C\varepsilon') <0
\]
since $\varepsilon'$ can be made arbitrarily small by our choice of $\varepsilon$. It follows that $\partial g(\alpha,\gamma)/ \partial \gamma <0$ in $(\alpha^2,\alpha]$ and the interval does not contain any stationary points.

When $\alpha \ge \varepsilon \ln(d)/d$, there may be a second stationary point near $(\alpha,\alpha)$ and the values of $g$ at the stationary points must be compared to show that $(\alpha,\alpha^2)$ is the global maximum. This will be done in the points that follow.

\item  We show that there is a constant $c_2 >c_1$ so that for  
  $\gamma \in (\alpha^2,\gamma_2]$, $\frac{\partial g}{\partial \gamma}<0$. Thus, for this range of $\gamma$ there are no
  stationary points of $g$ (see Figure \ref{fig:del-g}).
Integrating the second derivative in this range we obtain that

\begin{align*}
\frac{\partial g}{\partial \gamma}(\alpha,\gamma_2)
&  = \int_{\alpha^2}^{\gamma_1} \frac{\partial^2 g(\alpha,s)}{\partial s^2}
ds+ \int_{\gamma_1}^{\gamma_2} \frac{\partial^2 g(\alpha,s)}{\partial s^2}
ds \nonumber 
\le  -\int_{\alpha^2}^{\gamma_1} \frac{1}{2s} ds +
\int_{\gamma_1}^{\gamma_2} Cd ds\\
& =  \frac{1}{2} \ln\left( \frac{\alpha^2}{\gamma_1} \right)+
Cd(\gamma_2-\gamma_1) 
 \le  \frac{1}{2} \ln\left( \frac{\alpha\ln(\alpha^{-1})}{c_1} \right)+
C(c_2-c_1).
\end{align*}
The upper bounds in the first line above follow by the arguments similar to those of 1) and 2) above. In particular, $c_1$ can be chosen small enough so that the first bound follows. The last inequality follows since $\alpha<\alpha_c$ which implies $d< 3\alpha^{-1}\ln(\alpha^{-1})$. As $d \to \infty$, $\alpha \to 0$ and
therefore for large enough $d$, the
first derivative will be negative as claimed.

\item There are constants $c_2,c_3$ such that $c_2 > c_1 $ and for $\gamma \in \left[\gamma_2,\alpha-\gamma_3\right]$, $\frac{\partial^2 g}{\partial \gamma^2} >0$ (see Figure \ref{fig:del2-g}). This  implies that $g$ does not have a maximum in this range.

For this range of $\gamma$, the first term of $\frac{\partial^2 g}{\partial \gamma^2}$ in \eqref{eq:d2g-by-dgamma2} can be bounded as
\begin{align}
-\left(
\frac{2}{\alpha-\gamma}+\frac{1}{\gamma} + \frac{1}{1-2\alpha+\gamma}
\right) \ge -\left(\frac{2}{c_3}+\frac{1}{c_2}\right)\alpha^{-1}\ln(\alpha^{-1}).
\label{eq:del2-firstterm}
\end{align}

The second term can be bounded below as follows. We use the Taylor series expansion
\[
\sqrt{x+y} = \sqrt{x} + \sum_{i=1}^\infty
\displaystyle x^{\frac{1}{2} - i} y^i
\left(\prod_{k=1}^{i}\left(\frac{3}{2}-k\right)\frac{1}{i!} \right)
\]
to expand as follows with $x=(1-2\alpha)^2$ and $y=4(\alpha-\gamma)^2$:
\[
\overline\varepsilon = (\alpha-\gamma) - \frac{1}{2}\sum_{i=1}^\infty
(1-2\alpha)^{1 - 2i} (2(\alpha-\gamma))^{2i} \left(
  \displaystyle\prod_{k=1}^{i}\left(\frac{3}{2}-k\right)\frac{1}{i!} \right).
\]
Rearranging, we have
\[
\alpha-\gamma-\overline\varepsilon =  \frac{1}{2} \sum_{i=1}^\infty
 \frac{(2(\alpha-\gamma))^{2i}}{(1-2\alpha)^{2i-1}} \left(
  \displaystyle\prod_{k=1}^{i}\left(\frac{3}{2}-k\right)\frac{1}{i!}
\right).
\]
Similarly, 
\[
\sqrt{(1-2\alpha)^2 +4(\alpha-\gamma)^2} =  \sum_{i=0}^\infty
 \frac{(2(\alpha-\gamma))^{2i}}{(1-2\alpha)^{2i-1}} \left(
  \displaystyle\prod_{k=1}^{i}\left(\frac{3}{2}-k\right)\frac{1}{i!}
\right).
\]
Therefore, the second term of \eqref{eq:d2g-by-dgamma2} can be bounded as follows
\begin{align}
& d\left( \frac{1}{1-2\alpha+\gamma}+ \frac{2}{\alpha-\gamma}-
\frac{2(\alpha-\gamma)}{(\alpha-\gamma-\overline\varepsilon)
  \sqrt{(1-2\alpha)^2+4(\alpha-\gamma)^2}} \right) \nonumber \\
& = d \left(1+ \frac{2}{\alpha-\gamma}\left(   1- \frac{1}{\left(
  \frac{1}{1-2\alpha} -
  \frac{(\alpha-\gamma)^2}{(1-2\alpha)^3} + \cdots \right)
  \left((1-2\alpha) + \frac{2(\alpha-\gamma)^2}{1-2\alpha} -
  \frac{2(\alpha-\gamma)^4}{(1-2\alpha)^3}+ \cdots\right)} \right)\right) \nonumber \\
& =   d\left(1+ \frac{2(\alpha-\gamma)}{(1-2\alpha)^2}
\frac{1- 4 \frac{(\alpha-\gamma)^2}{(1-2\alpha)^2} \cdots}{1+ \frac{(\alpha-\gamma)^2}{(1-2\alpha)^2} - 4 \frac{(\alpha-\gamma)^4}{(1-2\alpha)^4} \cdots } \right)
\ge d\left(1+ \frac{2(\alpha-\gamma)}{(1-2\alpha)^2} \left( 1-  5 \frac{(\alpha-\gamma)^2}{(1-2\alpha)^2}\right) \right) \ge d  
\label{eq:del2-secondterm}
\end{align}
Combining \eqref{eq:del2-firstterm} and \eqref{eq:del2-secondterm}
we see that to prove the claim it is enough to choose constants $c_2$ and $c_3$ to satisfy $\left(\frac{2}{c_3}+\frac{1}{c_2}\right) \alpha^{-1}\ln(\alpha^{-1}) < d.$ Recall that we are now in the case that $d> \varepsilon' \alpha^{-1}\ln(\alpha^{-1})$. Thus, the claim follows by choosing $c_2>c_1$ and $c_2$ and $c_3$ large enough.

\item Lastly, we show that for $\gamma \in [\alpha-\gamma_3,\alpha]$,
  the maximum obtained has a smaller value than the maximum at
  $\gamma=\alpha^2$ (see Figure \ref{fig:g}).
By \eqref{eq:d2g-by-dgamma2} we have
\begin{align*}
\frac{\partial g^2}{\partial \gamma^2} \ge - \frac{2}{\alpha-\gamma} \left( 1+ \frac{\alpha-\gamma}{\gamma} \right) \ge
\frac{-2\left(1+\frac{c_3}{\ln(\alpha^{-1})-c_3}\right)}{\alpha-\gamma}
\end{align*} 

Suppose that $\gamma_4$ is any maximum of $g$ in this interval  so that
$\partial g(\alpha,\gamma)/\partial \gamma=0$. We can bound the first derivative as follows.
\begin{align*}
\frac{\partial g(\alpha,\gamma)}{\partial \gamma} = \int_{\gamma_4}^{\gamma}\frac{\partial^2 g(\alpha,s)}{\partial s^2}ds & \ge
-2\left(1+\frac{c_3}{\ln(\alpha^{-1})-c_3}\right)\int_{\gamma_4}^{\gamma}
\frac{1}{\alpha-s}ds   \\
& = -2\left(1+\frac{c_3}{\ln(\alpha^{-1})-c_3}\right) \ln
\left(\frac{\alpha-\gamma_4}{\alpha-\gamma}\right)
\end{align*}

Now, using the above bound, we integrate to obtain
\begin{align*}
& g(\alpha,\alpha)-g(\alpha,\gamma_4) = \int_{\gamma_4}^{\alpha}\frac{\partial g(\alpha,s)}{\partial s}ds \ge -2\left(1+\frac{c_3}{\ln(\alpha^{-1})-c_3}\right)
(\alpha-\gamma_4) >
-2(1+\varepsilon_d)(\alpha-\gamma_4) \\
 \Rightarrow & \ \ \ \ \ \ \ \ g(\alpha,\gamma_4) < g(\alpha,\alpha) +2(1+\varepsilon_d)(\alpha-\gamma_4)
\end{align*}
where $\varepsilon_d = 1/(\ln d - \ln \ln d)$.
We would like to show that $g(\alpha,\alpha^2) > g(\alpha,\gamma_4)$. Thus it suffices to show that $g(\alpha,\alpha^2) - g(\alpha,\alpha) >
  2(1+\varepsilon_d)(\alpha-\gamma_4)$.
Roughly, $g(\alpha,\alpha)$ should behave like $\Phi$, the first moment, while $g(\alpha,\alpha^2)$ is $2\Phi$. Recall that $g(\alpha,\alpha) = f(\alpha,\alpha,0)$ since $\overline \varepsilon(\alpha,\alpha) =0$. Comparing \eqref{eq:first-moment-product} and \eqref{eq:second-moment-product}, we have
\begin{align*}
\exp\left( n f(\alpha,\alpha,0)\right) =
\lambda^{\alpha n}e^{n\Phi(\alpha)}  \\
\Rightarrow g(\alpha,\alpha) = f(\alpha,\alpha,0) = \alpha
\log(\lambda) + \Phi(\alpha)
\end{align*}
Also,
\begin{align*}
\Phi(\alpha) = \alpha \log(\lambda) + \frac{1}{n}\ln(|\{ \sigma : |\sigma| = \alpha n\}|)
\end{align*}
Therefore,
\begin{align}
g(\alpha,\alpha^2)-g(\alpha,\alpha)& = 2\Phi(\alpha) - \alpha\log(\lambda) -
\Phi(\alpha) \nonumber \\
 & = \frac{1}{n}\ln(|\{ \sigma : |\sigma| = \alpha n\}|) \nonumber \\
& =   H(\alpha) + d \left(
(1-\alpha)\ln(1-\alpha) + (\alpha - \frac{1}{2})\ln(1-2\alpha)
  \right). \label{eq:gasuare-ga}
\end{align}
Using Mathematica to expand the terms of the expression in \eqref{eq:gasuare-ga}, we obtain for $\alpha < \al_c$
\begin{align*}
 H(\alpha)  + d \left(
(1-\alpha)\ln(1-\alpha) + (\alpha - \frac{1}{2})\ln(1-2\alpha)
  \right)  & =(1-\ln(\alpha))\alpha - \frac{d+1}{2}\alpha^2 -
  \left(\frac{d}{2}+\frac 16 \right)\alpha^3 +O(\alpha^4) \\
& >2(1+\varepsilon_d)\alpha   >  2(1+\varepsilon_d)(\alpha -\gamma_4)\\
\end{align*}
for $d$ sufficiently large. 

\end{enumerate}
The first part of Proposition \ref{prop:f-global-max} follows by the five facts above. Next we show that $f(\alpha,\gamma,\varepsilon)$ decays quadratically near $\alpha^*,\gamma^*,\varepsilon^*$.
We start by calculating the Hessian matrix for $f$.
The partial second mixed derivatives of $f(\alpha,\gamma,\varepsilon)$ evaluated at $\hat \gamma = \alpha^2$ and $\hat \varepsilon = \alpha(1-2\alpha)$ are given by:
\begin{align*}
 & \frac{\partial^2 f}{\partial \alpha^2}  = \frac{\alpha^2 (6-21 d)+2 \alpha^4 (-4+d)-d+16 \alpha^3 d+\alpha (-2+8 d)}{(1-2 \alpha)^2 (-1+\alpha)^2 \alpha^2}\\
 & \frac{\partial^2 f}{\partial \alpha \partial \gamma} =\frac{\alpha (2-4 d)+d+\alpha^2 d}{(-1+\alpha)^2 \alpha^2}\\
&  \frac{\partial^2 f}{\partial \alpha \partial \varepsilon} = \frac{\left(1-4 \alpha+2 \alpha^2\right) d}{(1-2 \alpha)^2 \alpha^2}\\
 %
%
& \frac{\partial^2 f}{\partial \gamma^2}  = \frac{-1+\left(-1+4 \alpha-2 \alpha^2\right) d}{(-1+\alpha)^2 \alpha^2}\\
& \frac{\partial^2 f}{\partial \gamma \partial \varepsilon}  = -\frac{d}{\alpha^2}\\
%
%
%
&  \frac{\partial^2 f}{\partial \varepsilon^2}  = -\frac{\left(1-2 \alpha+2 \alpha^2\right) d}{(1-2 \alpha)^2 \alpha^2}
\end{align*}
Using Mathematica to calculate the characteristic polynomial of the Hessian matrix, we obtain
\begin{align*}
& \frac{1}{(1-\alpha)^3 \alpha^5 (1-2 \alpha)^3} \big((2 d-8 \alpha d+12 \alpha^2 d-8 \alpha^3 d+2 \alpha d^2-6 \alpha^2 d^2+8 \alpha^3 d^2-2 \alpha^3 d^3) \\
& + (2 \alpha^2 -12 \alpha^3 +24 \alpha^4 -16 \alpha^5 +2 \alpha d -14 \alpha^2 d +50 \alpha^3 d -102 \alpha^4 d +120 \alpha^5 d -92 \alpha^6 d +40 \alpha^7 d \\  & \; \; \; \;  \; \; \; \; \; \; \; \; +2 \alpha^3 d^2 -10 \alpha^4 d^2 +10 \alpha^5 d^2 +16 \alpha^6 d^2 -20 \alpha^7 d^2 )x\\
& + (\alpha^3 -5 \alpha^4 +6 \alpha^5 +2 \alpha^6 +4 \alpha^7 -24 \alpha^8 +16 \alpha^9 +3 \alpha^3 d -29 \alpha^4 d +116 \alpha^5 d -236 \alpha^6 d +246 \alpha^7 d \\ & \; \;  \; \;  \; \; \; \; \; \; \; \; -116 \alpha^8 d +16 \alpha^9 d  )x^2\\
& + (\alpha^5 -9 \alpha^6 +33 \alpha^7 -63 \alpha^8 +66 \alpha^9 -36 \alpha^{10} +8 \alpha^{11} )x^3 \big)
\end{align*}
Recall that $ \alpha = O(\ln d / d) $. Thus by taking $d$ sufficiently large, it is enough to consider only the leading terms in the coefficients of the characteristic polynomial. Recall also that $\alpha\le1/2$. First, note that the value of the polynomial at $x=0$ which is given by the constant coefficient is positive. Secondly, each of the other coefficients is also positive, so that for all $x>0$, the derivative of the characteristic polynomial is positive. Hence the polynomial itself has only real negative roots and the Hessian is negative definite. Thus, $f$ is strictly concave at $\alpha^*,\gamma^*,\varepsilon^*$
and must decay quadratically around its maximum. Hence we have that for any $(\alpha,\gamma,\varepsilon)$
\[
f(\alpha^*,\gamma^*,\varepsilon^*) - f(\alpha,\gamma,\varepsilon)  \ge  C (|\alpha-\alpha^*|^2 + |\gamma-\gamma^*|^2+ |\varepsilon-\varepsilon^*|^2). \qedhere
\]
\end{proof}

We conclude by establishing Proposition \ref{prop:second-moment-square-of-first} which says that second moment of the partition function can be bounded by the  by the square of the first moment up to a polynomial term.

\begin{proof}[Proof of Proposition \ref{prop:second-moment-square-of-first}]
Applying the Cauchy-Schwartz inequality  we have
\begin{align*}
\E((Z_G)^2) & =\sum_{\alpha,\alpha'} \E\left(  Z_{G,\alpha} Z_{G,\alpha'}\right)
 = \tilde \Theta(1) \sum_{\alpha} \E ( (Z_{G,\alpha})^2)\\
 & = \tilde \Theta(1)  \sum_{\alpha,\gamma,\epsilon} \E  Z^{(2)}_{G,\alpha,\gamma,\varepsilon} \leq \exp(n f(\alpha^*,\gamma^*,\epsilon^*)+O(\log n))\\
&=   \exp(2\Phi(\alpha^*) n +O(\log n)) =\tilde \Theta(1)  (\E Z_G)^2. \qedhere
\end{align*}
where the second inequality is by Proposition \ref{prop:f-approximates-Z-square} while the penultimate equality is by Lemma~\eqref{lem:second-moment-square-of-first}.
\end{proof}

\bibliographystyle{plain}
\bibliography{all,my,extra}

\end{document}